[1994/12/01]

\documentclass{amsart}
\usepackage{amssymb}
\usepackage{amsmath,amssymb,amsfonts,amsthm,graphics,
latexsym, amscd, amsfonts, epsfig, eepic,epic}
\usepackage{mathrsfs}
\usepackage{algorithm}
\usepackage{algorithmic}
\usepackage{color}

\usepackage{eucal}
\usepackage{eufrak}
\usepackage[all]{xypic}
\usepackage{xspace}

\newtheorem{thm}{Theorem}[section]
\newtheorem{lem}[thm]{Lemma}
\newtheorem{prop}[thm]{Proposition}
\newtheorem{cor}[thm]{Corollary}

\newtheorem{ques}[thm]{Question}
\newtheorem{conj}[thm]{Conjecture}

\theoremstyle{definition}
\newtheorem{defn}{Definition}[section]

\theoremstyle{remark}
\newtheorem{rem}{Remark}[section]

\makeatletter \@addtoreset{equation}{section}

\newcommand{\thmref}[1]{Theorem~\ref{#1}}

\newcommand{\lemref}[1]{Lemma~\ref{#1}}

\def\a{\alpha}
\def\b{\beta}

\def\l{\lambda}

\def\p{\partial}
\def\vphi{\varphi}

\def\L{\Lambda}

\def\and{\quad{\rm and}\quad}

\def\eps{\epsilon}

\let\lra=\longrightarrow

\def\mapright{\xrightarrow}
\def\mapright\#1{\,\smash{\mathop{\lra}\limits^{\#1}}\,}

\def\om{\omega}
\def\Om{\Omega}

\def\tri{\triangle}
\def\ul{\underline}
\newcommand{\gij}{\ensuremath{g_{i\bar j}}}
\newcommand{\as}{\underline{S}}

\numberwithin{equation}{section}

\DeclareMathOperator{\osc}{Osc}

\makeatletter

\newcommand{\Rmnum}[1]{\expandafter\@slowromancap\romannumeral #1@}

\title{Pseudo-Calabi Flow}
\author[Chen]{Xiuxiong Chen}
  \address{Department of Mathematics, University of Wisconsin, Madison, WI 53706, USA}
  \email{xiuxiong.chen@gmail.com}
\author[Zheng]{Kai Zheng}
  \address{Academy of Mathematics and Systems Sciences, Chinese Academy of Sciences, Beijing, 100190, P.R.~China}
  \email{kaizheng@amss.ac.cn}
\date{}
\begin{document}
\maketitle

\section{abstract}
We first define
Pseudo-Calabi flow, as
\begin{equation*}
  \left\{
   \begin{aligned}
   {{\partial \varphi}\over {\partial t}}&= -f(\varphi),  \\
   \triangle_\varphi f(\varphi) &= S(\varphi) - \ul S.  \\
   \end{aligned}
  \right.
\end{equation*}
Then we prove the well-posedness of this flow including the short time existence, the regularity of the solution and the continuous dependence on the initial data. Next, we point out that the $L^\infty$ bound on Ricci curvature
is an obstruction to
the extension of the pseudo-Calabi flow. Finally, we show that
if there is a cscK metric in its K\"ahler class,
then for any initial
potential in a small $C^{2,\alpha}$ neighborhood of it,
the pseudo-Calabi flow
must converge exponentially to a nearby cscK metric.
\tableofcontents

\section{Introduction}

A renowned problem in K\"ahler geometry is
to find a constant scalar curvature K\"ahler metric (cscK)
in an arbitrary K\"ahler class.
When restricted to the canonical K\"ahler class,
a cscK metric is nothing but a K\"ahler-Einstein metric (KE).
The cscK metric satisfies a (totally nonlinear) 4th order partial differential equation
of the K\"ahler potential,
and the usual variational method is difficult to apply. Calabi \cite{MR645743}\cite{MR780039} suggested to
use the heat flow method:
\[
   {{\p \varphi}\over {\p t}} = S(\varphi) - \as,
\]
which has become known in the literature as the Calabi flow.
A critical point of this
Calabi flow is precisely a cscK metric.
This flow has been actively studied in recent
years (c.f.\cite{MR1101689}\cite{MR1820328}\cite{MR2405167}\cite{chen-2007}\cite{chen-2008}\cite{MR2357473}...).
However,
the remaining technical difficulties are still daunting simply
because it is a 4th order flow. \\

One wonders if we can take "square root" of the Calabi flow.
If so, then it will become a 2nd order flow and we can reduce it to something we are all familiar
with.
What we propose here is a very natural approach (where
we call it Pseudo-Calabi flow):

\begin{equation*}
  \left\{
   \begin{aligned}
   {{\partial \varphi}\over {\partial t}}&= -f(\varphi),  \\
   \triangle_\varphi f(\varphi) &= S(\varphi) - \ul S.  \\
   \end{aligned}
  \right.
\end{equation*}

In the canonical K\"ahler class, this flow will reduce to the famous K\"ahler-Ricci flow.
If we start with a metric
in general K\"ahler class, the leading term on
the right hand side is the same as the K\"ahler-Ricci flow
(the logarithm of volume
ratio between the evolving K\"ahler metrics
and the fixed reference metric). However,
there is an additional, unpleasant term,
which is 0th order pseudo-differential operator.
It is this additional term which makes things very subtle.
Nonetheless, it is crucial that
the cscK metric is the fixed point of the pseudo-Calabi flow. Moreover, in the space of K\"ahler metrics equipped with the Calabi gradient metric, the pseudo-Calabi flow is the gradient flow of the $K$-energy (see Remark \ref{notations: gradient flow}).

We first prove the following theorem:

\begin{thm} If the initial K\"ahler potential is
in $C^{2,\alpha} (0<\alpha <1)$,
then the flow exists for short time.
More importantly,
it becomes smooth right after $t>0.\;$
\end{thm}

Following the corresponding work in the Calabi flow
(\cite{MR2405167}...),
we have

\begin{thm}\label{intro: ricci} The $L^\infty$ bound of the Ricci curvature
is the obstruction to
an extension of the pseudo-Calabi flow.
\end{thm}

In order to make a case that the pseudo-Calabi flow is the right approach,
we need the
following stability theorem.

\begin{thm}\label{intro: stability} If there is a cscK metric in its K\"ahler class,
then for any initial
potential in a small $C^{2,\alpha}$ neighborhood of it,
the pseudo-Calabi flow
must converge exponentially to a nearby cscK metric.
\end{thm}

\noindent {\bf Remark} There is an important earlier work in geometric flow which is essential a variant of this flow.
In a very interesting work \cite{MR2163555}, Simanca considered the so called ``extremal flow"
$\frac{\p\varphi}{\p t}=-G_t(S_t-\pi_tS_t)=F(\vphi)$,
where
$\pi_t$ is $L^2$-orthogonal projection operator onto the space of
real holomorphic potentials.  In a sense, it is a slight variation of the ``Pseudo-Calabi flow", which
we consider here, by a lower order term. One of main motivations for such a modification is that
fixed points of the extremal flow
$S=\pi S$ are precisely extremal metrics. We believe that our version has a simpler concept and
it can also be viewed as a generalization of
the K\"ahler-Ricci flow to the non-canonical K\"ahler class,
since it agrees with the K\"ahler-Ricci flow in the canonical class. Simanca proved the short time existence of the extremal flow
for the $G$-invariant initial K\"ahler potentials in the
space $C_{(k+1,0)}$ with the norm $||\varphi||=\sup_{t\in
I}\{\sup_{0\leq r\leq k+1}||\p_t^rv||_{W_G^{2(k+1-r),2}}\}$
provided that $2k>n+2.\;$
First of all, he chose an approximate solution in a small time interval by solving heat equation with the given $G$-invariant initial data.
Secondly, he used semigroup method to obtain the solution $v$ of the linearized equation with the coefficient which is determined by the approximate solution.
Thirdly, following Kato's program in \cite{MR930267}, he defined a map by solving $\l \vphi - F(\vphi)=-v(t)+\l(\vphi_0+\int_0^t v(s)ds )$ for some real number $\l$, such that $\l-F'_{\vphi_0}$ is an isomorphism. Finally, he obtained
the fixed point of this map and solved the extremal flow.
It is a nice work indeed. However, the proof
is difficult to comprehend (from the standard PDE point of view), and our assumption on regularity of initial data is weaker. \\

Guan \cite{MR2366370} defined a modified Ricci flow
$\frac{\p}{\p t}g=-Ric(g)+HRic(g)+L_Vg$ where $HRic(g)$
is the harmonic part of the Ricci form
and $V$ is a real holomorphic vector field.
Then he considered the problem of finding
generalized quasi-Einstein metrics.
This flow is another complicated variation of the pseudo-Calabi flow. Guan claimed the short time existence of his flow
with $C^\infty(M)$ initial data
by a very brief outline. \\

We remark that, when $M$ admits no holomorphic vector field, these three flows coincide. However, the soliton solutions formed under these three flows are different. We do believe our flow is the simplest
which allow us to focus on the main challenges which are arisen from the geometric aspects of cscK metrics. \\


\noindent {\bf Acknowledgements:}
The second author would like to thank Professor W. Y. Ding for constant encouragements over the past several years. He would also like to
thank Professor S. B. Angenent
for introducing him his paper \cite{MR1059647} and Professor J. Y. Li, Professor J. Qing, Professor Y. D. Wang,
Professor M. J. Zhu and Professor X. H. Zhu
for their helps and encouragements.
Thanks also go to ICTP, USTC and UW-Madison
for their hospitality during the time when part of the work was done.

We wish to thank Professor S. Simanca
for carefully reading an earlier version of this manuscript and for his helpful comments.

\section{Notations and setup}\label{notations}

Let $M$ be a n-dimensional compact K\"ahler manifold.
$\om$ is a K\"ahler form belonging to a fixed K\"ahler class $\Omega$. In the local coordinates
$(z_1,z_2,\dotsm z_n)$, we have

\[
 \om =
\frac{\sqrt{-1}}{2}\sum\limits_{i=1}^{n} \gij dz^i \wedge
dz^{\bar j}.
\]
The Rimannian metric corresponding to $\om$ is given by
$ g =
\sum\limits_{i=1}^{n} \gij dz^i \otimes dz^{\bar j }$
on $T^\mathbb{C}(M)$.
Written in this form, the metric $g$ is K\"ahler
if and only if
\begin{align*}
g_{ij}=g_{\bar i\bar j}=0 \text{ and }
\frac{\p g_{i\bar j}}{\p z^k}=\frac{\p g_{k\bar j}}{\p z^i}.
\end{align*}
The volume form is the $(n,n)$ form
\begin{equation*}
dV=\omega^{[n]}
=\frac{\om^{n}}{n!}
={(\frac{\sqrt{-1}}{2})}^n\det(\gij)
dz^1\wedge dz^{\bar{1}} \wedge \dotsm \wedge dz^n\wedge dz^{\bar{n}}.
\end{equation*}
For each $\omega\in \Omega$,
the corresponding Ricci form
\begin{equation*}
Ric= \frac{\sqrt{-1}}{2} \sum\limits_{i=1}^{n}
R_{i\bar j} dz^i \wedge dz^{\bar{j}}
=- \frac {\sqrt{-1}}{2}\p \bar \p \log (\om^n)
\end{equation*}
is a closed form, in which $\log (\om^n)$ is generally not a
globally defined function on $M$. The first Chern class is
$C_1(M)=\frac{[Ric]}{\pi}$. The scalar curvature
is the contraction of the Ricci curvature $$S=g^{i\bar{j}} R_{i\bar{j}}$$
and the Futaki potential $f$ is the real value solution of
$$\tri_\vphi f=S-\as$$ with
$\int_{M}e^{f}\om^n_{\vphi}=V$.
Furthermore, since
$$S\om^{n}=n Ric \wedge \om^{n-1},$$
we obtain that the average of the scalar curvature is
\begin{equation*}
\ul S =
\frac{\int_{M}SdV}{V}
=\frac{1}{(n-1)!V}\int_{M}Ric\wedge\om^{n-1}
=\frac{Ric[\om]^{[n-1]}}{[\om]^{[n]}}
=\frac{\pi C_1(M)[\om]^{[n-1]}}{[\om]^{[n]}}
\end{equation*} which only depends on the K\"ahler class $[\om]$. The space of K\"ahler potentials is defined as
\[
\mathcal{H}
=\{\vphi\in C^{\infty}(M,R)\vert\om+\frac{\sqrt{-1}}{2}\p\bar\p\vphi>0\}.
\]
Donaldson \cite{MR1736211}, Mabuchi \cite{MR909015} and Semmes \cite{MR1165352}
defined a Riemannian metric as
\[
\int_M f_1 f_2 \om^n_\vphi
\]
for any $f_1,f_2\in T_\vphi \mathcal{H}$,
under which $\mathcal{H}$ becomes a non-positive curved
infinite dimensional symmetric space.
Chen \cite{MR1863016}
proved that any two points in $\mathcal{H}$
can be connected by a $C^{1,1}$ geodesics
and $\mathcal{H}$ is a metric space. Later, Calabi and Chen proved $\mathcal{H}$ is negatively curved in the sense of Alexanderof in \cite{MR1969662}.
The space of normarlized K\"ahler potentials is defined as
\[
\mathcal{H}_0
=\{\vphi\in C^{\infty}(M,R)\vert\om+\frac{\sqrt{-1}}{2}\p\bar\p\vphi>0
\text{ and }I(\vphi)=0\},
\]where
\begin{align}\label{notations: I}
I(\vphi)
=\sum_{p=0}^n\frac{1}{(p+1)!(n-p)!}
\int_M\om^{n-p}(\p\bar\p\vphi)^p\vphi.
\end{align}
Calabi \cite{Cacom} suggested another metric on $\mathcal{H}$,
\begin{align}\label{notations: gra Calabi metric}
\int_M g_\vphi^{i\bar{j}}f_{1i}f_{2\bar{j}} \om^n_\vphi
\end{align}
for any $f_1,f_2\in T_\vphi \mathcal{H}_0$. It is computed in Calamai \cite{simone-2010} that the Calabi's gradient metric on $\mathcal{H}_0$ admits a unique Levi-Civita connection.

The pseudo-Calabi flow is defined as
\begin{equation*}
\left\{
\begin{aligned}
\frac{\p}{\p t}\gij&=-f_{i\bar{j}},\\
g(0)&=g_0,\\
\end{aligned}
\right.
\end{equation*}
in the fixed but arbitrary K\"ahler class $\Om$.
According to the definition of the K\"ahler condition
we see that the pseudo-Calabi flow preserves the K\"ahler condition, i.e.
\begin{thm}
If $g_0$ is K\"ahler, then $g(t)$ is K\"ahler if $g(t)$ satisfies the pseudo-Calabi flow.
\end{thm}
The equation for the K\"ahler form is:
\begin{equation}\label{notations: form of PCF}
\left\{
\begin{aligned}
\frac{\p\om_{\vphi}}{\p t}
&=-\frac{\sqrt{-1}}{2}\p\bar\p f,\\
\om_{\vphi(0)}&=\om_0.\\
\end{aligned}
\right.
\end{equation}
We observe the following:
\begin{thm}\label{notations: pre kahler class}
The pseudo-Calabi flow preserves the K\"ahler class.
\end{thm}
We now show that, when the class $\Omega$ is the canonical class, the pseudo-Calabi flow is just the K\"ahler-Ricci flow.
First, we recall that the K\"ahler-Ricci flow is
\begin{align*}
\frac{\p g_{\vphi i\bar{j}}}{\p t}
=\l g_{\vphi i\bar{j}}-R_{i\bar{j}},
\end{align*} where $\l$
is the sign of the first Chern class.
Its potential equation is
\begin{align*}
\frac{\p \vphi}{\p t}&=h+\l\vphi-h_\om,
\end{align*}
where $h_\om$, the Ricci potential of the background metric, satisfies
\begin{align*}
Ric(\om)-\l\om=\frac{\sqrt{-1}}{2}\p\bar\p
h_\om
\end{align*}
with
$\int_Me^{h_\om}\om^n=V$
and
$\int_Me^{\frac{\p \vphi}{\p t}-\l\vphi+h_\om}\om^n
=V$.
Next by definition, we have
\begin{align*}
\tri_\vphi f
=S-\as=-\tri_\vphi(h+\l \vphi-h_\om).
\end{align*}
Then the maximum principle implies that
$f=-h-\l\vphi+h_\om$. Finally we conclude that $f_{i\bar{j}}=R_{i\bar j}-\l g_{\vphi i\bar j}$, i.e. the pseudo-Calabi flow in the canonical class coincides with the K\"ahler-Ricci flow.

An observation which is used in the sequel is that the pseudo-Calabi flow can be
written as a
pseudo-differential Monge-Amp\`{e}re flow.
The potential equation of our
flow is
\begin{equation}\label{notations: PCF}
\left\{
\begin{aligned}
\tri_\vphi\frac{\p \vphi}{\p
t}&=-\tri_\vphi f=-S_\vphi+\ul S,\\
\vphi(0)&=\vphi_0.\\
\end{aligned}
\right.
\end{equation}
According to \cite{MR2405167},
we have a decomposition of the scalar curvature as
\begin{equation*}
S_{\vphi}=-\tri_{\vphi} h+tr_{\vphi} Ric(\om),
\end{equation*}
in which
\begin{equation*}
h=\log\frac{\om_{\vphi}^n}{\om^n}=
\log\frac{\det(\gij+\vphi_{i\bar{j}})}{\det(\gij)}.
\end{equation*}
Therefore the pseudo-Calabi flow can be rewritten as
\begin{equation}\label{notations: pam of PCF}
\frac{\p\vphi}{\p t}=-f+c(t)=h-P+c(t),
\end{equation}
where
\begin{equation}\label{notations: P}
\tri_\vphi
P=tr_{\vphi}{Ric(\om)}-\as=
g_{\vphi}^{i\bar{j}}R_{i\bar{j}}(\om)-\as
\end{equation}
under the normalization condition
\begin{align}\label{notations: P normalization condition}
\int_Me^{P}\om^n=V.
\end{align}
The function $P$ is well defined since we have
$\int_M{tr_{\vphi}{Ric(\om)}\om_{\vphi}^n}=
\as[\om_{\vphi}]^{[n]}.$
Since the volume is invariant along the pseudo-Calabi flow, we can choose
\begin{align}\label{notations: normalization condition}
\int_Me^{\frac{\p\vphi}{\p t}+P}\om^n=V.
\end{align} such that $c(t)=0$.
Note that for any $c_1(t)$ and $c_2(t)$ defined by two different normalization conditions, the corresponding solutions of \eqref{notations: pam of PCF} with the same initial data only differ by the constant $\int^t_0c_1(s)-c_2(s)ds$.

\section{Energy functionals}\label{Enerfun}
In this section, we assume that $\vphi(t)$ is the $C^\infty(M,g)$
solution of the pseudo-Calabi flow.
According to \thmref{notations: pre kahler class}, we see that the flow keeps the volume fixed,
so the flow can be viewed as some
kind of "normalized" flow.
Since the flow preserves the K\"ahler class
and the average of the scalar curvature $\ul S$ is an invariant of the K\"ahler class,
$\ul S$ stays constant under the flow.
The most important observation here
is that the $K$-energy is decreasing along the flow,
which makes it reasonable to search for cscK metrics using the pseudo-Calabi flow.

The $K$-energy is defined by Mabuchi \cite{MR867064} as the following
\begin{align}\label{notations: K energy}
\nu_\om(\vphi)
&=-\frac{1}{V}
\int_0^1\int_{M}\dot\vphi(\tau)(S_{\vphi(\tau)}-\as)
\om^n_{\vphi(\tau)}d\tau\
\end{align}
where $\vphi(\tau)$ is a path from $0$ to
$\vphi$.
\begin{thm}
The $K$-energy decreases along the pseudo-Calabi flow. Furthermore,
the $t$-derivative of the $K$-energy
achieves zero at some t if and only if
$\om_{\vphi(t)}$ is a cscK metric.
\end{thm}
\begin{proof}Plugging \eqref{notations: PCF} in \eqref{notations: K energy}, we obtain the derivative of the $K$-energy along the flow:
\begin{align*}
\delta_{\dot\vphi}\nu_\om(\vphi)
=-\frac{1}{V}\int_{M}\dot\vphi(S-\as)\om^n_\vphi
=-\frac{1}{V}\int_{M}|\nabla f|_{g_\vphi}^2\om^n_\vphi.
\end{align*} So the $K$-energy decreases along the flow unless $\nabla f\equiv0$.
Hence $g(t_0)$ is cscK for some $t_0$, or $g(t)$ converges to a cscK metric as $t$ tends to infinity if the $K$-energy is bounded from below.
\end{proof}
\begin{rem}\label{notations: gradient flow}
Since the derivative of the $K$-energy at $\vphi$ is
\begin{align*}
\delta\nu_\om(\vphi)
=\frac{1}{V}\int_{M}g_\vphi^{i\bar j}\dot\vphi_i f_{\bar j}\om^n_\vphi,
\end{align*} we conclude that \eqref{notations: pam of PCF} is the gradient flow of the $K$-energy in the space $\mathcal{H}_0$ with the Calabi's gradient metric \eqref{notations: gra Calabi metric}.
\end{rem}

\section{Short time existence of the pseudo-Calabi flow}\label{short time}
Let $C^{2,\a}(M,g)$ be the completion of
the smooth function under the $C^{2,\a}$ norm.
It is called little H\"older space in customary literature.
We shall show that the Cauchy problem for
the pseudo-Calabi flow,
\begin{equation}\label{short time: pam of PCF}
\left\{
\begin{aligned}
\frac{\p}{\p t}\vphi&=h-P,\\
\tri_\vphi
P&=tr_{\vphi}{Ric(\om)}-\ul S,\\
\vphi(0)&=\vphi_0,\\
\end{aligned}
\right.
\end{equation}
with the normalization condition \eqref{notations: P normalization condition} and \eqref{notations: normalization condition} has a short time solution for any initial K\"ahler potential in $C^{2,\a}(M,g)$.

The proof of local existence with the $C^{2,\a}$ initial data is quite different from the case of smooth initial data. DaPrato-Grisvard \cite{MR551075}, Angenent \cite{MR1059647} and some other mathematicians developed the abstract theory of local existence for the fully nonlinear parabolic equation. They have \cite{MR551075} constructed the continuous interpolation spaces so that the linearized operator stays in certain class.

We apply their ideas to prove the short time existence.
\begin{itemize}
\item
We first linearize the fully nonlinear equation at the initial data.

\item
Next we derive a priori estimates, related to our special solution space, and apply it to prove the linearized equation has a local solution by using the contraction mapping theorem.

\item
Then using Remark \ref{short time: choose small initial}, a key decomposition of the initial data, we derive the energy inequality. Combining the former inequality with the Sobolev imbedding theorem and the bootstrap method, we obtain a priori estimates of the solution to the linearized equation.
Thus, the solution exists for all the time and its $C^{2,\a}$ norm is continuous in $t$.
\item
Finally we construct a sequence of solutions to the linear approximation equations,
and show that it is a contractive sequence by choosing small time or small initial data.
\end{itemize}

Let
$(x_1,\cdots,x_{2n})$ be the local real coordinate.
We fix a background K\"ahler metric $g\in C^{\infty}(M,g)$ in $[\om]$.
Let $Q_T=M\times[0,T]$ be the time-space.
The point and the
distance in $Q_T$ are denoted by $X=(x,t)$ and
$d(X,X_0)=(d(x,x_0)^2+|t-t_0|)^\frac{1}{2}$, respectively.
The H\"older spaces on $Q_T$ are defined as the following:
\begin{align*}
&C^{1,\a}(Q_T)
=C^{1+\a,\frac{1+\a}{2}}(Q_T)=\{\vphi\vert|
\vphi|_{1+\a,\frac{1+\a}{2}}=|\vphi|_{C^0(Q_T)}+|D_{x}
\vphi|_{C^0(Q_T)}\\
&+|D_x\vphi|_{C^\a(Q_T)}+\sup_{x\in M,t\neq s,t,s\in[0,T]}
\frac{|\vphi(x,t)-\vphi(x,s)|}{|t-s|^{\frac{1+\a}{2}}}<\infty\},\\
&C^{2,\a}(Q_T)=C^{2+\a,1+\frac{\a}{2}}(Q_T)=\{\vphi\vert|
\vphi|_{2+\a,1+\frac{\a}{2}}=|\vphi|_{C^0(Q_T)}+|D_x
\vphi|_{C^0(Q_T)}\\
&+|D_{xx}\vphi|_{C^\a(Q_T)}+|D_t\vphi|_{C^\a(Q_T)}\}.
\end{align*}
The Sobolev spaces are defined as
\begin{align*}
&W^{q,r}(Q_T)
=\{\vphi\vert[\int_0^T(\int_M|
\vphi|^q\om^n)^{\frac{r}{q}}dt]^{\frac{1}{r}}\},
W^{q,q}(Q_T)=L^q(Q_T),\\
&W_2^{2k,k}(Q_T)
=\{\vphi\in L^2(Q_T) \text{ with }
\sum_{0\leq|p|+2q\leq2k}||D^p_xD^q_t\vphi||_{L^2(M)}<\infty\},\\
&W^{1,0}_2(Q_T)
=\{\vphi\vert||\vphi||_{2,Q_T}+||\nabla\vphi||
_{2,Q_T}<\infty\},\\
&V_2(Q_T)
=\{\vphi\in
W^{1,0}_2(Q_T)\vert||\vphi||_{V_2(Q_T)}=\sup_{0\leq t \leq
T}||\vphi(t)||_{2;M}+||\nabla\vphi||_{2,Q_T}<\infty\}.
\end{align*}
We denote $$X^k_T=
C^0([0,T],C^{k+2+\a}(M,g))\cap
C^1([0,T],C^{k+\a}(M,g))$$ which equipped with the norm $||\cdot||_{X^k_T}=\max_{0\leq t\leq T}[||\p_t\cdot||_{C^{k,\a}}+||\cdot||_{C^{k+2,\a}}]$.
We also denote $$\dot{X}^k_T=
C^0([0,T),C^{k+2+\a}(M,g))\cap
C^1([0,T),C^{k+\a}(M,g)).$$
Here, all the derivatives and norms are defined with respect to the background
metric $\om$.
We shall prove the following results.
\begin{thm}\label{short time: main}
Suppose that $\vphi_0\in C^{2,\a}(M,g)$ satisfies $\l \om
\leq \om_{\vphi(0)}\leq \L \om$ for some positive constants $\l$ and $\L$.
Then under the normalization conditions \eqref{notations: P normalization condition} and \eqref{notations: normalization condition}, the Cauchy problem for
the pseudo-Calabi flow \eqref{short time: pam of PCF}
has a unique solution $$\vphi(x,t)\in \dot X^0_T,$$
where $T$ is the maximal existence time.
\end{thm}
\begin{rem}
If in addition $\vphi_0\in C^{k,\a}(M,g)$ we obtain
$\vphi(x,t)\in \dot X^k_T.$
\end{rem}
\begin{thm}\label{short time: stability}
Let $M$ admits a cscK metric $\om$. Suppose that $\vphi_0\in C^{2,\a}(M,g)$ satisfies $\l \om
\leq \om_{\vphi(0)}\leq \L \om$ for some positive constants $\l$ and $\L$.
Then for any $T>0$
there exits a positive constant $\eps_0(T,g)$. If
$|\vphi_0|_{C^{2,\a}(M,g)}\leq\eps_0$,
then the equation has a unique solution on $[0,T]$, and
$$||\vphi||_{X^0_T}\leq C\eps_0,$$ where $C$ depends on $M$, $g$ and $T$. Furthermore
$\eps_0$ goes to zero, as $T$ goes to infinity.
\end{thm}
\begin{rem}
The imbedding theorem \eqref{regularity of PCF: imb lem} implies that the solutions in both theorems satisfy $D_{ij}\vphi\in C^{\a,\frac{\a}{2}}(M\times [0,T))$. Then
using the equation \eqref{short time: pam of PCF} we obtain
$\vphi\in C^{2+\a,1+\frac{\a}{2}}(M\times[0,T))$.
\end{rem}
Chen's conjecture \cite{Chen4-2008} says that a global $C^{1,1}$ $K$-energy minimizer in any K\"ahler class must be smooth. This conjecture has been proved in the canonical K\"ahler class via the weak K\"ahler-Ricci flow \cite{MR2348984}\cite{MR2434691}\cite{ctz-2008}\cite{songtian2009}. We hope that the pseudo-Calabi flow will be the right approach to solve this conjecture. In Subsection \ref{estimate of linearized}, we obtain a partial estimates related to this conjecture.

\begin{rem}\label{short time: choose small initial}
For any $\vphi_0\in C^{2,\a}$, we can
choose a smooth function $\bar{\phi}$ which approximates $\vphi_0$ in
$C^{2,\a}$ norm. Let $\phi=\bar\phi-\int_M\bar\phi\om^n$.
Then we replace the reference metric $\om$ and $\vphi$ by
$\tilde{\om}=\om+\frac{\sqrt{-1}}{2}\p\bar\p\phi$ and
$\tilde{\vphi}_0=\vphi_0-\phi$ respectively in the equation \eqref{short time: pam of PCF}, so that
\begin{equation*}
  \left\{
   \begin{aligned}
   \frac{\p\tilde{\vphi}}{\p t}&=\log\frac{\det
   (\tilde{g}_{i\bar{j}}+\tilde{\vphi}_{i\bar{j}})}{\det(\tilde
   {g}_{i\bar{j}})}-P(\tilde{\vphi}),\\
   \tri_{\tilde{\vphi}}
   P(\tilde{\vphi})&=g_{\tilde{\vphi}}^{i\bar{j}}R_{i\bar{j}}(\om)-\ul S,\\
   \tilde{\vphi}(0)&=\tilde{\vphi}_0
   =\vphi_0-\phi\in{C^{2,\a}(M)}.\\
   \end{aligned}
  \right.
\end{equation*}
It is obvious that $\tilde{\vphi}+\phi$ gives the solution to the original equation.
Here $|\tilde{\vphi}_0|_{C^{2,\a}(M)}$ could be small enough to
be used later.
Moreover, we have
$\tilde{\om}=\om+\frac{\sqrt{-1}}{2}\p\bar\p\phi>0$,
since
$\om_{\vphi_0}-\tilde\om=\frac{\sqrt{-1}}{2}\p\bar\p\tilde\vphi_0$
is sufficiently small.
\end{rem}
\begin{proof}(proof of \thmref{short time: main})
To prove the theorem we employ the idea of the inverse function theorem in
\cite{MR0488101}\cite{MR656198}
and its adaption to a parabolic equation in \cite{MR1059647}\cite{MR551075}.

We introduce the following notations
\begin{align*}
J&=[0,T],
E_1=C^{2,\a}(M),
E_0=C^\a(M),\\
X&=\{\vphi\in C^1(J,E_0)\cap
C^0(J,E_1)\vert
||\vphi||_X\\
&=\max_{[0,T]}
(||\p_t\vphi||_{C^{\a}(M)}
+||\vphi||_{C^{2,\a}(M)})<\infty\},\\
Y&=\{\psi=(\vphi_1,\vphi_2)\in C^0(J,E_0)\times E_1\vert
||\psi||_Y\\&=\max_{[0,T]}||\vphi_1||_{C^{\a}(M)}
+||\vphi_2||_{C^{2,\a}(M)}<\infty\}.
\end{align*}

Let $U=\{\vphi\in E_1\vert\om_{\vphi}>0\}$. Then $U$ is an open subset of $E_1$ and $\vphi_0\in U$.
We define the map $\Phi$ from $U$
 to
$Y$ as the following
\begin{align}\label{short time: mapping}
\Phi:C^0(J,U)\cap C^1(J,E_0)&\rightarrow C^0(J,E_0)\times U,\\
\vphi&\mapsto(\p_t\vphi-h+P,\vphi_0)\nonumber.
\end{align}
Note that the pseudo term $P(\vphi)$ belongs to $C^0(J,E_0)$
by \lemref{short time: P regular} which
we will prove below.
Moreover, \lemref{short time: P regular}
assures $\Phi$ is infinite
Fr\'{e}chet differentiable at any point of $U$.
Choosing the approximate solution
\begin{align}\label{app sol fix}
\bar\vphi(t)\equiv\vphi_0
\text{ on  }[0,T],
\end{align} we get
\begin{align}\label{short time: app va}
\Phi(\bar\vphi)
=(-h(\vphi_0)+P(\vphi_0),\vphi_0)
\triangleq(\bar\psi,\vphi_0)\triangleq\psi_0.
\end{align}

Now we compute the linearized equation in a usual way.
Let the variation of $\vphi$ be
$$\dot\vphi\triangleq\frac{\p}{\p s}(\vphi+sv)|_{s=0}=v.$$
Differentiating both sides of \eqref{notations: P},
we get $$ \dot \tri_\vphi P+\tri_\vphi \dot
P=-v^{i\bar{j}}R_{i\bar{j}}(\om).$$
Thus we obtain
\begin{align*}
\tri_\vphi \dot P=v^{i\bar{j}}[P_{i\bar{j}}-R_{i\bar{j}}(\om)]
=-v^{i\bar{j}}T_{i\bar{j}}.
\end{align*}
Here the $(1,1)$-tensor
\begin{align}\label{short time: harmonic tensor}
T_{i\bar{j}}=-P_{i\bar{j}}+R_{i\bar{j}}(\om)
\end{align} is a
harmonic tensor which is smooth when $\vphi\in C^{3}$.
Since $\Phi$
is at least $C^1$ Fr\'{e}chet differentiable in $U$,
we get
\begin{align*}
\Phi'(\bar\vphi): C^0(J,E_1)&\rightarrow C^0(J,E_0)\times E_1,\\
v&\mapsto(\frac{\p v}{\p t}-\tri_{\bar\vphi}
v-Q,v_0).
\end{align*}
Here the function $Q$ satisfies the following equation
\begin{align}\label{short time: Q}
\tri_{\bar\vphi}
Q=v^{i\bar{j}}T_{i\bar{j}}
=-v^{i\bar{j}}(P_{i\bar{j}}-R_{i\bar{j}}(\om))
\end{align}
with the normalization condition
\begin{align}\label{short time: Q normalization condition}
\int_MQe^{P(\bar\vphi)}\om^n=0
\end{align} by differentiating \eqref{notations: P normalization condition}. Then $Q$ belongs to $C^0(J,E_0)$,
according to \lemref{short time: P regular}. By using
the following Proposition~\ref{short time: Linearized iso} for special case $\bar{\vphi}=\vphi_0$, we deduce that
the linearized
operator $\Phi'(\bar{\vphi})$ is a linear isomorphism.

In order to
seek a unique solution $\vphi$ satisfying the flow equation
$\Phi(\vphi)=(0,\vphi_0)$,
firstly we introduce the closed ball of $C^0(J,U)$ by $$B(0,\eps)=\{\rho\in C^0(J,U)\vert\max_{[0,T]}||\rho||_{C^{2,\a}(M)}\leq\eps\}$$ for some $\eps$ determined later. Then for any $\rho\in B(0,\eps)$ we define a map $A$ by
\begin{align}\label{short time: rho}
A\rho
=[\Phi'(\bar\vphi)]^{-1}[\psi-\Phi(\bar\vphi+\rho)]+\rho.
\end{align}
Letting $\rho_0=0$, we obtain a sequence by $\rho_{n+1}=A\rho_n$ inductively.
Secondly, we show that ${\rho_n}$ is a contractive sequence.
For any $\rho_1$ and $\rho_2$ in $B(0,\eps)$, we compute
\begin{align*}
A\rho_1-A\rho_2
=[\Phi'(\bar\vphi)]^{-1}[\int_0^1[\Phi'(\bar\vphi+s\rho_2+(1-s)\rho_1)
-\Phi'(\bar\vphi)](\rho_2-\rho_1)]ds.
\end{align*}
Since $\Phi$ is a $C^1$ map on open subset $U$, there exists a small constant $\eps$ such that
$$||\Phi'(\bar\vphi+s\rho_2+(1-s)\rho_1)-\Phi'(\bar\vphi)
||\leq\frac{1}{2||[\Phi'(\bar\vphi)]^{-1}||}.$$
Consequently, we obtain
\begin{align}\label{short time: contract}
\max_{[0,T]}||A\rho_1-A\rho_2||_{C^{2,\a}}
\leq\frac{1}{2}\max_{[0,T]}||\rho_2-\rho_1||_{C^{2,\a}}.
\end{align}
Moreover, for all $T'\leq T$, $\varepsilon$ is uniform
since
$||[\Phi'(\bar\vphi)]^{-1}||\geq\frac{1}{||\Phi'(\bar\vphi)||}$
is uniformly bounded below. Thirdly, we verify that $A$ maps $B(0,\varepsilon)$ into itself.
By the triangle inequality and \eqref{short time: contract} we have
\begin{align*}
\max_{[0,T]}||A\rho||_{C^{2,\a}}
&\leq\max_{[0,T]}||A\rho-A\rho_0||_{C^{2,\a}}+\max_{[0,T]}||A\rho_0||_{C^{2,\a}}\\
&\leq\frac{1}{2}\varepsilon+\max_{[0,T]}||A(0)||_{C^{2,\a}}.
\end{align*}
To estimate the second term, we use \eqref{short time: app va} and \eqref{short time: rho} to obtain
$$A(0)=[\Phi'(\bar\vphi)]^{-1}[(0,\vphi_0)-\Phi(\bar\vphi)]
=[\Phi'(\bar\vphi)]^{-1} (h(\vphi_0)-P(\vphi_0),0)\in X.$$
According to Proposition~\ref{short time: Linearized iso}, $A(0)$ solves the
following equations,
\begin{equation*}
  \left\{
   \begin{aligned}
\frac{\p\bar v}{\p t}-\tri_
{\bar\vphi}\bar{v}-Q(\bar{v})&=h(\vphi_0)-P(\vphi_0) \text{ in } Q_T,\\
\tri_{\bar\vphi} Q(\bar{v})&=\bar v^
{i\bar{j}}T_{i\bar{j}} \text{ in } Q_T,\\
\bar{v}(0)&=0.\\
   \end{aligned}
  \right.
\end{equation*}
Since $\bar{v}$ is continuous in $t$ by \lemref{estimate of linearized: con in t}, one could
choose a time $T_2$ small enough such that
\begin{align}\label{short time: T2}
\max_{[0,T_2]}||\bar{v}||_{C^{2,\a}(M)}\leq\frac{\varepsilon}{2}.
\end{align}

Note that $T_2\leq T$,
after we replace the time $T$ by $T_2$ in all of the argument above, both
the spaces $X$, $Y$, and the mappings $\Phi$
and $\Phi'(\bar\vphi)$ may change.
However the result in
Proposition~\ref{short time: Linearized iso}
and the contraction property of the mapping
$A$ remain valid.
Meanwhile, since the injectivity of the linearized operator implies
$[\Phi'(\bar\vphi)]_{T_2}^{-1}[\psi-\Phi(\bar\vphi)|_{T_2}]
=\bar{v}|_{T_2}$, the
following inequity holds
\begin{align*}
\max_{[0,T_2]}||A(0)||_{C^{2,\a}(M)}
&=\max_{[0,T_2]}||[\Phi'(\bar\vphi)]_{T_2}^
{-1}[\psi-\Phi(\bar\vphi)|_{T_2}]||_{C^{2,\a}(M)}
\leq\frac{1}{2}\varepsilon.
\end{align*}

Hence, $A$ is a contractive mapping and it maps $B(0,\varepsilon)$ into itself. So the contraction mapping theorem implies that there do exists a fixed point
$\rho\in B(0,\varepsilon)$ such that
$\psi=\Phi(\vphi_0+\rho)\in C^0(J,U)$ is a solution of \eqref{short time: pam of PCF} on
the small lifespan $[0,T_2]$.
Furthermore, we deduce $\vphi_0+\rho\in U$
from the equation \eqref{short time: pam of PCF}.
Therefore the main \thmref{short time: main} follows
by solving the flow equation with the initial data, given by
the value of the solution
at the end of the previous interval,
till the maximal existence time $T$.
\end{proof}
\begin{rem}
In fact, we can produce the approximate solution by using the following Monge-Amp\`{e}re type flow
\begin{equation}\label{short time: app sol}
  \left\{
   \begin{aligned}
\frac{\p\vphi}{\p t}
&=\log\frac{\om^n_\vphi}{\om^n}-P(\vphi_0),\\
\vphi(0)&=\vphi_0,  \\
   \end{aligned}
  \right.
\end{equation}
since Cao \cite{MR799272} proved the long time existence of such flow.
\end{rem}

\begin{proof}(proof of \thmref{short time: stability})
In order to solve $\Phi(\vphi)=(0,\vphi_0)$,
we only need to verify the following inequality
\begin{align*}
\max_{[0,T]}||A(0)||_{C^{2,\a}(M)}
&=\max_{[0,T]}||[\Phi'(0)]^{-1}(0,\vphi_0)||_{C^{2,\a}(M)}
\leq\frac{1}{2}\varepsilon.
\end{align*}
So it suffices to choose $\vphi_0$ such that $|\vphi_0|_{C^{2,\a}}$
is a small constant depending on $T$.
Thus, the rest of the proof of \thmref{short time: stability} follows
along the same lines of the rest of the proof of \thmref{short time: main}.
\end{proof}
\subsection{Continuous dependence on initial data}\label{con dpd}
The continuous dependence on initial data of solutions is used to study the stability of the pseudo-Calabi flow near a cscK metric which is the equilibrium solution of this flow.
\begin{thm}\label{con dpd: main}
If $\phi$ is a solution of \eqref{short time: pam of PCF} for initial data $\phi_0$ on $[0,T]$, then there is a neighborhood $U$ of $\phi_0$ such that \eqref{short time: pam of PCF} has a solution $\vphi(t)$ on $[0,T]$ for any $\vphi_0\in U$ and the mapping $\vphi_0\mapsto\vphi(t)$ is $C^k$ for $k=0,1,2,\ldots$
\end{thm}
\begin{proof}
We derive, substituting $\psi=\vphi-\vphi_0$ into the potential equation \eqref{short time: main},
\begin{equation*}
\left\{
\begin{aligned}
\frac{\p}{\p t}\psi&=\log\frac{\om^n_{\psi+\vphi_0}}{\om^n}-P(\psi+\vphi_0),\\
\tri_{\psi+\vphi_0} P(\psi+\vphi_0)&=tr_{\psi+\vphi_0}{Ric(\om)}-\ul S,\\
\psi(0)&=0.\\
\end{aligned}
\right.
\end{equation*}
It is obvious that its solution plus $\vphi_0$ gives the solution of \eqref{short time: main}.
Analogously to the mapping \eqref{short time: mapping}, we define a mapping as
\begin{align*}
\Phi:C^1(J,E_0)\cap C^0(J,U)\times U&\rightarrow C^0(J,E_0),\\
(\psi,\vphi_0)&\mapsto\p_t\psi
-\log\frac{\om^n_{\psi+\vphi_0}}{\om^n}+P(\psi+\vphi_0).
\end{align*}
Then we have $\Phi(\phi-\phi_0,\phi_0)=0$ and
the Fr\'{e}chet derivative with respect to $\psi$ at
$(\phi-\phi_0,\phi_0)$ is given by
\begin{align*}
\Phi'(\phi-\phi_0,\phi_0):C^1(J,E_0)\cap C^0(J,E_1)\times E_1&\rightarrow C^0(J,E_0),\\
(v,w)&\mapsto\frac{\p v}{\p t}-\tri_{\phi} v-Q(\phi).
\end{align*}
We apply Remark~\ref{short time: choose small initial} and
Proposition~\ref{short time: Linearized iso} with $\vphi=\phi$. Hence
$\Phi'(\phi-\phi_0,\phi_0)$ is an isomorphism
and the theorem follows from the implicit function theorem.
\end{proof}

\subsection{The main proposition}
From here on we shall simply write $\vphi$
for both approximate solutions $\bar{\vphi}$ and $\phi$.
We are going to prove the following proposition for $\vphi$; we stress that $\vphi$ may depend on $t$.
Due to Remark \ref{short time: choose small initial},
we can further assume that $\vphi$ satisfies
\begin{align}\label{short time: vphi small}
\max_{[0,T]}|\vphi(t)|_{C^{2,\a}}\leq\delta\ll1
\end{align} for some $\delta$ to be determined later. Generally, we introduce the space
\begin{align}\label{short time: v delta}
V_\delta=\{\psi\in C^0([0,T],U)\vert \max_{[0,T]}|\psi(t)|_{C^{2,\a}}\leq\delta\}.
\end{align}

\begin{prop}\label{short time: Linearized iso}
Under the normalization condition
\eqref{short time: Q normalization condition} the linearized equation
\begin{equation}\label{short time: linearized equ}
  \left\{
   \begin{aligned}
\frac{\p v}{\p t}-\tri_{\vphi}v-Q
&=u(x,t) \text{ in } Q_T,\\
\tri_{\vphi} Q(v)&=v^
{i\bar{j}}(R_{i\bar{j}}(\om)-P_{i\bar{j}}) \text{ in } Q_T,\\
v(0)&=w, \\
   \end{aligned}
  \right.
\end{equation}
has a unique solution $v(x,t)\in X$, for any $u(x,t)\in
C^0([0,T],C^\a(M))$ and $w\in C^{2,\a}(M)$.
\end{prop}
Since the linearized operator is not
self-adjoint, we can not use the Fredholm theory directly. In order to
prove the Proposition~\ref{short time: Linearized iso},
we first show that the linearized
equation has a short time solution by using the contraction mapping
theorem. Then we prove that the solution of the linearized
equation exists for any $t$.
Before demonstrating the procedure, we
need the following technical lemmas.
\subsection{The technical lemmas}
We first deal with
the pseudo-differential term $P$ which satisfies \eqref{notations: P} and \eqref{notations: P normalization condition}.
The following identity is computed directly.
\begin{lem}
Suppose that $\vphi_1$ and $\vphi_2$ are two $C^{2,\a}$ K\"ahler potentials. Let $g_a=ag_{\vphi_2}+(1-a)g_{\vphi_1}$. We have
\begin{align}\label{short time: P diff}
\tri_{\vphi_1}(P(\vphi_1)-P(\vphi_2))=-\int_0^1g^{i\bar l}_ag^{k \bar j}_ad(\vphi_1-\vphi_2)_{k\bar l}(R_{i\bar j}-P(\vphi_2)_{i\bar j}).
\end{align}
\end{lem}
The regularity of $P$ in the space direction can be improved
due to the regularity of $\vphi$.
\begin{lem}\label{short time: P regular}For any $\vphi(x,t)\in C^0([0,T],U)$, we have
\begin{align*}
P\in C^0([0,T],C^{2+\a}(M))
\end{align*}
and
\begin{align*}
T_{i\bar{j}}=-P_{i\bar{j}}+R_{i\bar{j}}(\om)\in
C^0([0,T],C^{\a}(M)).
\end{align*}
\end{lem}
\begin{proof}
Since $\vphi(x,t)\in C^0([0,T],U)$, we have
$$tr_{\vphi}{Ric(\om)}-\as\in
C^0([0,T],C^{\a}(M)).$$
It follows directly from the Schauder estimate that $$ P\in
L^\infty([0,T],C^{2+\a}(M)).$$ Let
$$g_{\vphi(t_s)}=sg_{\vphi(t_1)}+(1-s)g_{\vphi(t_2)}.$$ Then
we plug $\vphi_i=\vphi(t_i)$ for $i=1,2$ in \eqref{short time: P diff} to get
\begin{align*}
\tri_{\vphi(t_1)}(P(t_1)-P(t_2))
=-\int_0^1g_{\vphi(t_s)}^{k\bar{j}}g_{\vphi
(t_s)}^{i\bar{l}}ds(\vphi(t_1)-\vphi(t_2))_{k\bar{l}}[R_{i\bar{j}}(\om)
-P(t_2)_{i\bar{j}}].
\end{align*}
Freezing the coefficient by the fixed metric $\om$, we get
\begin{align}\label{short time: diff t }
\tri(P(t_1)-P(t_2))&=(\tri-\tri_{\vphi(t_1)})(P(t_1)-P(t_2))\\
&-\int_0^1g_{\vphi(t_s)}^{k\bar{j}}g_{\vphi
(t_s)}^{i\bar{l}}ds(\vphi(t_1)-\vphi(t_2))_{k\bar{l}}[R_{i\bar{j}}(\om)
-P(t_2)_{i\bar{j}}]\nonumber.
\end{align}
Let $f=P(t_1)-P(t_2)$. Then the Green representation gives
\begin{align*}
f-\int_Mf\om^n=-\frac{1}{V}\int_M\tri f\cdot G(g)(x,y)\om^n(y)
\end{align*}
which implies
\begin{align}\label{short time: nor diff t t}
|f-\int_Mf\om^n|\leq \delta|D^2f|+ C|D^2(\vphi(t_1)-\vphi(t_2))|
\end{align}
by condition \eqref{short time: vphi small}.
Since the normalization condition \eqref{notations: P normalization condition} implies both $\sup_M P$ and $-\inf_M P$ are nonnegative, we infer that
\begin{align*}
&0\leq \sup_M f\leq \int_Mf\om^n+\delta|D^2f|+ C|D^2(\vphi(t_1)-\vphi(t_2))|,\\
&0\geq \inf_M f\geq \int_Mf\om^n-\delta|D^2f|- C|D^2(\vphi(t_1)-\vphi(t_2))|\nonumber.
\end{align*}
Then it follows that
$$-\delta|D^2f|- C|D^2(\vphi(t_1)-\vphi(t_2))|\leq\int_Mf\om^n\leq \delta|D^2f|+ C|D^2(\vphi(t_1)-\vphi(t_2))|.$$
So by \eqref{short time: nor diff t t} we get
\begin{align*}
|f|\leq 2\delta|D^2f|+C|D^2(\vphi(t_1)-\vphi(t_2))|.
\end{align*}
By applying the Schauder estimate to \eqref{short time: diff t } again,
we deduce that
\begin{align*}
|f|_{C^{2,\a}(M)}\leq C(\delta|D^2f|+|D^2(\vphi(t_1)-\vphi(t_2))|
+|D^2(\vphi(t_1)-\vphi(t_2))|_{C^{\a}(M)}).
\end{align*} Let $\delta$ be small enough so that
\begin{align*}
|f|_{C^{2,\a}(M)}\leq C(|D^2(\vphi(t_1)-\vphi(t_2))|
+|D^2(\vphi(t_1)-\vphi(t_2))|_{C^{\a}(M)}).
\end{align*}
Then $|P(t)|_{C^{2,\a}}$ is continuous with respect to $t$, since $\vphi\in C^0([0,T],U)$.
\end{proof}
Now we construct a sequence of smooth K\"ahler potentials approximating $\vphi$ in $C^0([0,T],C^{2,\a}(M))$.
\begin{lem}\label{short time: app coe}
For any $C^{2,\a}(M)$ K\"ahler potential $\vphi_0$, there exists a sequence of
smooth K\"ahler potentials $\vphi_n$ such that
$$\lim_{n\rightarrow\infty}|\vphi_n-\vphi_0|_{C^{2,\a}(M)}=0.$$
Moreover, the corresponding harmonic tensors
$$\lim_{n\rightarrow\infty}|T_{ni\bar j}-T_{0i\bar j}|_{C^{\a}(M)}=0.$$
\end{lem}
\begin{proof}
Let $\phi_0=\tri\vphi_0$. We select a sequence of smooth functions $\phi_n$ such that it converges to $\phi_0$ in $C^{\a}(M)$. Let $\vphi_n$ be the smooth solution of $\tri\vphi_n=\phi_n$ with the normalization condition $\int_M\vphi_n\om^n=0$. The Schauder estimate implies
\begin{align*}
|\vphi_n-\vphi_0|_{C^{2,\a}(M)}\leq C(|\phi_n-\phi_0|_{C^{\a}(M)}+|\vphi_n-\vphi_0|_{C^{0}(M)}).
\end{align*}
Meanwhile, the Green representation gives the following $L^\infty$ bound of $\vphi_n$
\begin{align*}
|\vphi_n-\vphi_0|_{C^{0}(M)}\leq C(|\phi_n-\phi_0|_{C^{0}(M)}).
\end{align*}
Then by combining these two estimates, we have $\om_{\vphi_n}$ is a K\"ahler form and the first part of the lemma follows.
The second part of the lemma follows from
\lemref{short time: P regular} by replacing
$P(t_1)$ and $P(t_2)$ with $P(\vphi_n)$ and
$P(\vphi_0)$ respectively in \eqref{short time: diff t }.
\end{proof}
\begin{lem}\label{short time: app coe t}
For any $\vphi(x,t)\in C^0([0,T],U)$, there exists a sequence of
smooth K\"ahler potentials $\vphi_n$ such that
$$\lim_{n\rightarrow\infty}\max_{[0,T]}|\vphi_n-\vphi_0|_{C^{2,\a}(M)}=0.$$
Moreover, the corresponding harmonic tensors satisfy
$$\lim_{n\rightarrow\infty}\max_{[0,T]}|T_{ni\bar j}-T_{0i\bar j}|_{C^{\a}(M)}=0.$$
\end{lem}
\begin{proof}
Similar to \lemref{short time: app coe}
$\vphi_n\in C^\infty(M\times[0,T])$ is obtained.
The rest part of the lemma follows from
\lemref{short time: P regular}.
\end{proof}
For lacking of maximal principle, the following lemmas play an important role in solving the linearized equation $Q$ satisfying
\eqref{short time: Q} and \eqref{short time: Q normalization condition}.
Let $H_0^p(M,g_\vphi)=\{\eta\in H^p(M,g_\vphi)\vert\int_M\eta\om^n_\vphi=0\}$.
Analogously to $P$, the function $Q$ is characterized by the following lemmas.
\begin{lem}\label{short time: Q diff}
Suppose that $\vphi_1$ and $\vphi_2$ are two $C^{2,\a}$ K\"ahler potentials. Let $g_a=ag_{\vphi_2}+(1-a)g_{\vphi_1}$. We have for any $C^2$ functions $v_1$ and $v_2$
\begin{align*}
&\tri_{\vphi_1}(Q(v_1,\vphi_1)-Q(v_2,\vphi_2))
=(-v_1^{i\bar j}+v_2^{i\bar j})(P_{i\bar j}(\vphi_2)
-R_{i\bar j}(\om))\\
&-v_1^{i\bar j}(P(\vphi_1)-P(\vphi_2))_{i\bar j}
-\int_0^1g_a^{i\bar l}g_a^{k\bar j}(\vphi_2-\vphi_1)_{k\bar l}
Q(\vphi_2)_{i\bar j}da.
\end{align*}
\end{lem}
Similarly to \lemref{short time: P regular}, we have
\begin{lem}\label{short time: Q regular}
For any $\vphi(x,t)\in C^0([0,T],U)$ and any $v\in C^0([0,T],C^{2+\a}(M))$, we have
\begin{align*}
Q\in C^0([0,T],C^{2+\a}(M)).
\end{align*}
\end{lem}
The following lemma enables us to derive the $L^p$ bound of $Q$.
\begin{lem}\label{short time: rho}
Suppose that $\vphi\in U$.
For any $\eta\in H_0^0(M,g_\vphi)$, let $\rho$ be the solution of $\tri_{\vphi}
\rho=\eta$ given by the $L^{p}$ theory. Then we
have
\begin{align*}
||\rho^{i\bar{j}}||_{2;M;g_\vphi}=||\eta||_{2;M;g_\vphi},
\end{align*}
and for all $p>1$, there exists
\begin{align*}
||\rho^{i\bar{j}}||_{p;M;g_\vphi}\leq C||\eta||_{p;M;g_\vphi}.
\end{align*}
\end{lem}
\begin{proof}
We compute for $p=2$
\begin{align*}
||\rho^{i\bar{j}}||_{2;M;g_\vphi}
=\int_M|\tri_\vphi\rho|^2\om^n_\vphi
=||\eta||_{2;M;g_\vphi}.
\end{align*} For $p>2$
the lemma follows from the $L^p$ estimate.
\end{proof}
\begin{lem}\label{short time: L2 estimate of Q}
Suppose that $\vphi\in U$.
If $v\in L^2(M)$,
then there exists a weak solution $Q\in L^{2}(M)$ of \eqref{short time: Q} in the distribution sense,
for any $\zeta\in C^\infty(M)$
\begin{align*}
\int_M\tri_\vphi\zeta Q\om_\vphi^n
=\int_MvT_{i\bar{j}}\zeta^{i\bar{j}}\om_\vphi^n.
\end{align*}
Moreover, for any $\eta\in L^2(M)$
\begin{align*}
\int_M\eta Q\om_\vphi^n\leq|T|_{0;Q_T}||v||_{2;M}||\eta||_{2;M}.
\end{align*}
\end{lem}
\begin{proof}
According to \lemref{short time: app coe t} there exist three sequences of smooth $\vphi_n$, $T_{ni\bar{j}}$ and $v_n$ such that $\vphi_{n}\rightarrow \vphi \text{ in }
C^0([0,T],C^{2,\a})$, $T_{ni\bar{j}}\rightarrow T_{i\bar{j}}\text{ in }
C^0([0,T],C^{\a})$ and $v_{n}\rightarrow v\text{ in } L^2(M)$.
Let $Q_n$ be the smooth solution of the equation
$$\tri_\vphi
Q_n=v_n^{i\bar{j}}T_{ni\bar{j}}=(v_nT_{ni\bar{j}})^{i\bar{j}}$$
and $\rho$ be the $W^{2,2}(M)$ solution of
$$\tri_\vphi\rho=\eta-\frac{1}{V}\int_M\eta\om^n_\vphi.$$
Note that we have
\begin{align*}
\int_M\eta (Q_n-\frac{1}{V}\int_MQ_n\om^n_\vphi)\om_\vphi^n
=\int_M\rho(v_nT_{ni\bar{j}})^{i\bar{j}}\om^n_{\vphi}
=\int_M\rho^{i\bar{j}}v_nT_{ni\bar{j}}\om^n_{\vphi}.
\end{align*}
Then by invoking \lemref{short time: rho} and the H\"older inequality, we have
\begin{align*}
\int_M\eta (Q_n-\frac{1}{V}\int_MQ_n\om^n_\vphi)\om_\vphi^n
&\leq||\rho^{i\bar{j}}||_{2;M,g_\vphi}||v_nT_{ni\bar{j}}||_{2;M,g_\vphi}\\
&\leq C|T_n|_{0;Q_T}||\eta||_{2;M}||v_n||_{2;M}.
\end{align*}
It follows that $Q_n-\frac{1}{V}\int_MQ_n\om^n_\vphi$
weakly converges to some function $Q_1$ in $L^2(M)$.
Furthermore, since for any $\zeta\in C^\infty(M)$
\begin{align*}
\int_M\tri_\vphi\zeta Q_n\om_\vphi^n
=\int_Mv_nT_{n i\bar{j}}\zeta^{i\bar{j}}\om_\vphi^n,
\end{align*}
after taking the limit we obtain the desired weak solution $Q=Q_1-\frac{\int_MQ_1e^{P(\vphi)}\om^n}
{\int_Me^{P(\vphi)}\om^n}\in L^2(M)$, whence the lemma follows.
\end{proof}

By using the $L^p$ estimate instead of the $L^2$ estimate, we can deduce that
\begin{lem}\label{short time: Lp estimate of Q}
Suppose that $\vphi\in C^0([0,T],U)$.
If $v\in L^{p}(Q_T)$,
then there exists a weak solution $Q\in L^{p}(Q_T)$
of \eqref{short time: Q} in the distribution sense,
for any $\zeta\in C^\infty(Q_T)$
\begin{align*}
\int_{Q_T}\tri_\vphi\zeta Q\om_\vphi^ndt
=\int_{Q_T}vT_{i\bar{j}}\zeta^{i\bar{j}}\om_\vphi^ndt.
\end{align*}
Moreover, for any $\eta\in L^q(Q_T)$ with
$p,q>1$ such that $\frac{1}{p}+\frac{1}{q}=1$, we have
\begin{align*}
\int_{Q_T}\eta
Q\om_\vphi^n\leq C|T|_{0;Q_T}||v||_{p;Q_T}||\eta||_{q;Q_T}.
\end{align*}
\end{lem}
Now we generalize a priori estimates for parabolic equation in a bounded domain to an arbitrary Riemannian manifold $M$.
We introduce some function spaces first.
Consider the following space $$B([0,t_0],C^{k,\a}(M))=\{v(t)\in C^0(M),
\forall t\in[0,t_0]\vert \sup_{[0,t_0]}||v||_{C^{k,\a}}<\infty\}.$$
Define also $$C^{\a,0}(M\times [0,t_0])
=\{v\in C^0(M\times [0,t_0])\vert v(t,\cdot)\in C^{\a}(M),
\forall t\in [0,t_0]\};$$ equip this space with the norm
$$\|v\|_{C^{\a,0}(M\times [0,t_0])}
=\sup_{0\leq{t}\leq{t_0}}|v(t)|_{C^{2,\a}(M)}.$$ Finally consider the space
$$C^{2+\a,1}(M\times [0,t_0])
=\{v\in C^{2,1}(M\times [0,t_0])\vert D_tv,D_{ij}v\in C^{\a}(M),
\forall t\in [0,t_0]\}$$
endowed with the norm
$$\|v\|_{C^{2+\a,1}(M\times [0,t_0])}
=||v||_\infty+\sum_{i=1}^n||D_iv||_\infty+
||D_tv||_{C^{\a,0}}+\sum_{i,j=1}^n||D_{ij}v||_{C^{\a,0}}.$$
We cite an imbedding Lemma 5.1.1 in Lunardi's book
\cite{MR1329547}.
\begin{lem}(Lunardi
\cite{MR1329547})\label{regularity of PCF: imb lem}
If $v\in C^{2+\a,1}(M\times [0,t_0])$, then
$D_{ij}v\in C^{\a,\frac{\a}{2}}(M\times [0,t_0])$.
\end{lem}
The following optimal regularity theorem in a domain $\Om$
is a part of Theorem 5.1.13 in Lunardi's book \cite{MR1329547} for a linear parabolic equation
\begin{align}\label{regularity of PCF: linear parabolic}
u_t=a^{ij}u_{ij}+b^iu_i+cu+f.
\end{align}
\begin{thm}(Lunardi \cite{MR1329547})\label{regularity of PCF: space reg Rn}
Let $\Om$ be an open set in $R^n$ of uniformly
$C^{2,\a}$ boundary.
Let $f\in UC(\bar\Om\times[0,t_0])\cap C^{\a,0}(\bar\Om\times[0,t_0])$ be such that $f(t,x)=0$ for every $t\in[0,t_0]$ and $x\in \p\Om$.
Suppose that the coefficients $a^{ij}$, $b^i$ and $c$ belong to $C^{\a}(\bar\Om)$
with $0<\a< 1$, and $a^{ij}$ satisfies
 $\L|\xi|^2\geq a^{ij}\xi_i\xi_j\geq\l|\xi|^2>0
\text{ for any }\xi\in R^n\backslash\{0\}$. If $u_0\in C^0(\bar\Om)$,
then \eqref{regularity of PCF: linear parabolic} has a unique solution which belongs to $C([0,T]\times\Om)$.  For every $\eps\in (0,t_0]$ there is $C$ such that
\begin{align}\label{regularity of PCF: space reg Rn est weak}
||u||_{C^{2+\a,1}(\bar\Om\times[\eps,t_0])}
\leq \frac{C}{\eps^{\frac{\a}{2}+1}}(|u_0|_{L^\infty(\bar\Om)}
+|f|_{C^{\a,0}(\bar\Om\times[0,t_0])}).
\end{align} If also $u_0\in C^{2,\a}(\bar\Om)$, then there exists $C'$ such that
\begin{align}\label{regularity of PCF: space reg Rn est}
||u||_{C^{2+\a,1}(\bar\Om\times[0,t_0])}
\leq C'(|u_0|_{C^{2,\a}(\bar\Om)}
+|f|_{C^{\a,0}(\bar\Om\times[0,t_0])}).
\end{align}
\end{thm}
Since $M$ is compact, it can be covered by a finite number of charts $\{\Om_p\}$.
Let $\eta_p(t)\in C_0^\infty(\Om_p)$ be a partition of unity subordinate to the charts $\{\Om_p\}$.
We assume $u_0\in C^{2,\a}(\bar\Om)$ and consider the following equation for $\eta_p u$
\begin{align*}
&(\eta_p u)_t-a^{ij}(\eta_p u)_{ij}-b^i(\eta_p u)_i-c\eta_p u\\
&=\tilde{f}=\eta_p f-2a^{ij}[(\eta_p)_i u]_{j}-a^{ij}(\eta_p)_{ij} u
+b^iu(\eta_p)_i \text{ in } \Om_p.
\end{align*}
Since $\tilde{f}|_{\p \Om_p}=0$,
applying \thmref{regularity of PCF: space reg Rn}
with $\Om=\Om_p$, we have
\begin{align*}
||\eta_p u||_{C^{2+\a,1}(\bar\Om_p\times[0,t_0])}
\leq C(|u_0|_{L^\infty(M)}
+|\tilde{f}|_{C^{\a,0}(\bar\Om_p\times[0,t_0])}).
\end{align*}
Combining all estimates in each ball and using the H\"older inequality,
we obtain
\begin{align}\label{regularity of PCF: space reg weak}
||u||_{C^{2+\a,1}(M\times[0,t_0])}
\leq C(|u|_{C^0(M\times[0,t_0])}+|u_0|_{C^{2,\a}(M)}
+|f|_{C^{\a,0}(M\times[0,t_0])}).
\end{align}
For any $f\in C^{\a,0}(M\times [\eps,t_0])$ we can
choose a smooth sequence $f_i$ which converges to $f$ in $C^{\a,0}(M\times [\eps,t_0])$. Then
by the classical existence theory (see \cite{MR0241822}),
one has a unique solution $u_i\in C^{2+\a,1+\frac{\a}{2}}$
of \eqref{regularity of PCF: linear parabolic} for each $f_i$.
The maximum principle implies that
$|u|_{C(M\times[0,t_0])}$ has a uniform bound.
Consequently, one can derive the solution of \eqref{regularity of PCF: linear parabolic} from the compactness of $u_i$. The uniqueness of the solution also follows from the maximum principle. In conclusion, we obtain:
\begin{thm}\label{regularity of PCF: space reg}
Let $f\in UC(M\times[0,t_0])\cap C^{\a,0}(M\times[0,t_0])$
and $u_0\in C^{2,\a}(M)$.
If the coefficients $a^{ij}$, $b^i$ and $c$ belong to $C^{\a}(M)$
with $0<\a< 1$, and $a^{ij}$ satisfies
 $\L|\xi|^2\geq a^{ij}\xi_i\xi_j\geq\l|\xi|^2>0
\text{ for any }\xi\in R^n\backslash\{0\}$.
Then \eqref{regularity of PCF: linear parabolic} has a unique solution $u\in C^{2+\a,1}(M\times[0,t_0])$ on $M$, and
\begin{align*}
||u||_{C^{2+\a,1}(M\times[0,t_0])}
\leq C(|u_0|_{C^{2,\a}(M)}
+|f|_{C^{\a,0}(M\times[0,t_0])}).
\end{align*}
\end{thm}
\subsection{Short time existence of the linearized equation}\label{short time linearized}
We now prove the short time existence of the linearized equation by using the
contraction mapping theorem.
Since the complex Laplacian is a real operator, we prove all a apriori estimates in the real coordinates.

Since the coefficients of the leading terms depending on $t$ create some problems, we freeze these coefficients and analyze
the corresponding modified equation instead.
This approach is similar to \cite{MR0164306} but more complicated,
because of the pseudo-differential term in our equations.

Before discussing the proof we recall the definition and some features of the heat kernel on a compact Riemannian manifold $(M,g)$
in \cite{MR768584}\cite{MR1333601}\cite{Ding}\cite{MR0181836}.
Suppose that $\rho(\l)\in C^\infty([0,\infty))$ satisfies
\begin{equation*}
  \begin{cases}
    \rho(\l)=1,  \l<\frac{inj(M)}{4},\\
    \rho(\l)=0,  \l>\frac{inj(M)}{2}.\\
  \end{cases}
  \end{equation*}
Set $\rho(x,y)=\rho(d(x,y))$.
On $M$
there exists a complete orthonomal basis $\{f_k\}$ of $L^2(M)$,
consisting of the eigenfunctions $f_k$ of $\tri$ with eigenvalues $\l_k$.
Then the heat kernel is given by
$$H(x,y,t)=\sum_{k=0}^\infty e^{-\l_kt}f_k(x)f_k(y).$$

\begin{prop}(\cite{MR768584}\cite{MR1333601}\cite{Ding}\cite{MR0181836})\label{short time linearized: heat kernel}
As $t\rightarrow0$, the heat kernel has an asymptotic expansion
$$H(x,y,t)=\frac{1}{(2\sqrt{\pi})^n}t^{-\frac{n}{2}}\exp^{-\frac{d^2(x,y)}{4t}}(\rho\sum_{k=0}^Kt^k\phi_k+O(t^{K+1})),$$ where $K>\frac{n}{2}+2$. The expression is independent of $K$, and $\phi_k$ are fixed functions constructed by Minakshisundaram-Pleijel.
The heat kernel has the properties\textup{:}
\begin{enumerate}
\renewcommand{\labelenumi}{(\roman{enumi})}
\item $(\frac{\p}{\p{t}}-\tri_x)H(x,y,t)=0$.
\item $H(x,y,t)$ is smooth except $x=y$,
positive and symmetric in $x$ and y.
\item $\int_MH(x,y,t)\om^n(y)=1$ for any $t>0$ and any $x\in{M}$.
\item $\lim_{t\rightarrow{0}}{\int_MH(x,y,t)f(y)\om^n(y)=f(x)}$
for any $f\in{C^0(M)}$.
\item $|H(x,y,t)|\leq{C}t^{-\frac{-n}{2}}e^{-\frac{d^2(x,y)}{t}},$\\
$|\nabla_xH(x,y,t)|\leq{C}t^{-\frac{-(n+1)}{2}}e^{-\frac{d^2(x,y)}{t}}.$\\
 In particular, these inequalities imply \\ $|H(x,y,t)|\leq{C}t^{-\b}d^{-n+2\b},$\\
$|\nabla_xH(x,y,t)|\leq{C}t^{-\b}d^{-n-1+2\b},$ for
some constant C and any $\b\in(0,1)$.
\item The solution of $\frac{\p}{\p{t}}\vphi-\tri{\vphi}
=f\in{C^0(M\times(0,t_0])}$ with
$\vphi(t=0)=\vphi(0)\in{C^0(M)}$ is of the form
$$\vphi(x,t)=\int_MH(x,y,t)\vphi(y,0)\om^n(y)
+\int_0^t\int_M{H(x,y,t-s)f(y,s)}\om^n(y)ds$$
on $(0,t_0]\times{M}$.
\end{enumerate}
\end{prop}
Consider the modified equation
\begin{equation}\label{short time linearized: itr equ}
  \left\{
   \begin{aligned}
\frac{\p v}{\p t}-\tri v+\tri v-\tri_{\vphi}v-Q
&=u(x,t) \text{ in } Q_T,\\
\tri_{\vphi}Q&=v^{i\bar{j}}T_{i\bar{j}} \text{ in } Q_T,\\
v(0)&=w.  \\
   \end{aligned}
  \right.
\end{equation}
Let $\tilde{v}=v-w$. Then the solution of the original linearized equation is obtained, by adding $w$ to the solution of the following equations
\begin{equation*}
  \left\{
   \begin{aligned}
&\frac{\p \tilde{v}}{\p t}-\tri\tilde{v}+\tri\tilde{v}
-\tri_{\vphi}\tilde{v}-\tri_{\vphi}w-\tilde{Q}-W
=u(x,t) \text{ in } Q_T,\\
&\tri_\vphi \tilde{Q}=\tilde{v}^{i\bar j}T_{i\bar j}\text{ in } Q_T,\\
&\tri_\vphi W=w^{i\bar j}T_{i\bar j}\text{ in } Q_T,\\
&\tilde{v}(0)=0.  \\
   \end{aligned}
  \right.
\end{equation*}
So we can solve
\eqref{short time linearized: itr equ} with $w=0$ instead.
Introduce the function space
\begin{align*}
Z=\{v\in C^{2+\a,0}(M\times [0,t_0])\vert v(0)=0\}
\end{align*}
with the norm
$\|v\|=\sup_{0\leq{t}\leq{t_0}}|v(t)|_{C^{2,\a}(M)}$.
We define the following operator:
\begin{multline}\label{short time linearized: app sol con}
\hat v=T(v)
=\int_0^t\int_M{H(x,y,t-s)\{Q+u-\tri v+\tri_\vphi v\}(y,s)}\om^n(y)ds.
\end{multline}
Let $\tilde{f}=\int_0^t\int_M{H(x,y,t-s)f(y,s)\om^n(y)ds}$.
\begin{lem}\label{short time linearized: C0 norm of app sol}
If $f\in  C^{0}(M\times [0,t_0])$,
then
\begin{align*}
|\tilde{f}|_{C^{0}(M)}
\leq{Ct_0^{1-\b}}\sup_{0\leq{t}\leq{t_0}}|f|_{C^{0}(M)}.
\end{align*}
\end{lem}
\begin{proof}
From Proposition~\ref{short time linearized: heat kernel} (\textrm{v}),
it follows directly
\begin{align*}
&|\int_0^t\int_M{H(x,y,t-s)f(y,s)\om^n(y)ds}|_{C^{0}(M)}\\
&\leq\int_0^t\int_MH(x,y,t-s)\om^nds\sup_{0\leq{t}\leq{t_0}}|f|_{C^{0}(M)}\\
&\leq Ct_0^{1-\b}\sup_{0\leq{t}\leq{t_0}}|f|_{C^{0}(M)}.
\end{align*}
\end{proof}
\begin{lem}\label{short time linearized: map itself}
$T$ is a map from $Z$ to itself.
\end{lem}
\begin{proof}
Since $v\in Z$, the Schauder estimate implies $Q\in C^{2+\a,0}(M\times [0,t_0])$.
Hence, we obtain that $Tv$ belongs to
$C^{2+\a,1}(M\times [0,t_0])$ and $Tv(0)=0$
from \thmref{regularity of PCF: space reg}.
\end{proof}
\begin{lem}
$T$ is a contraction map.
\end{lem}
\begin{proof}
We compute for any $v_1,v_2\in Z$
\begin{align*}
|T(v_1)-T(v_2)|_{C^{2,\a}(M)}&=|\int_0^t\int_M[H(x,y,t-s)
\{Q(v_1)-Q(v_2)\\
&+(\tri_{\vphi}-\tri)(v_1-v_2)\}
(y,s)]\om^n(y)ds|_{C^{2,\a}(M)}.
\end{align*}
According to \lemref{short time linearized: C0 norm of app sol} and \thmref{regularity of PCF: space reg}, we obtain
\begin{align*}
|T(v_1)-T(v_2)|_{C^{2,\a}(M)}
&\leq{C_1}t_0^{1-\b}
\sup_{0\leq{s}\leq{t_0}}|Q(v_1)-Q(v_2)|_{C^{2,\a}(M)}\\
&+C_2\sup_{0\leq{s}\leq{t_0}}|(\tri_{\vphi}-\tri)(v_1-v_2)|_{C^{\a}(M)}.
\end{align*}
Since $Q(v_1)-Q(v_2)$ satisfies
\begin{align*}
\tri_\vphi [Q(v_1)-Q(v_2)]=(v_1-v_2)^{i\bar{j}}T_{i\bar{j}},
\end{align*}
by applying \lemref{short time: P regular} we have that the first term can not exceed
\begin{align*}
{C_3}t_0^{1-\b}
\sup_{0\leq{s}\leq{t_0}}|v_1-v_2|_{C^{2,\a}(M)}.
\end{align*}
By using the following identity
\begin{equation*}
(\tri_{\vphi}-\tri)(v_1-v_2)=-\int_0^1g_{s\vphi}^{k\bar{j}}
g_{s\vphi}^{i\bar{l}}ds\vphi_{i\bar{j}}(v_1-v_2)_{k\bar{l}}
\end{equation*}
and the condition \eqref{short time: vphi small}, we see that the second term can be estimated as
\begin{align*}
|(\tri_{\vphi}-\tri)(v_1-v_2)|_{C^{\a}(M)}
&\leq C_4\delta|v_1-v_2|_{C^{2,\a}(M)}.
\end{align*}
Here $\delta$ becomes smaller while $t_0$ goes to $0$.
Adding all these estimates, we have
\begin{align*}
|T(v_1)-T(v_2)|_{C^{2,\a}(M)}
&\leq({C_3}t_0^{1-\b}+C_5\delta)
\sup_{0\leq{s}\leq{t_0}}|v_1-v_2|_{C^{2,\a}(M)}.
\end{align*}
Accordingly if $t_0$ and $\delta$ are small enough, $T$ is a contraction with constant $\frac{1}{2}$.
\end{proof}
Therefore we get a fixed point $v$ that satisfies
$Tv=v\in Z$ by the contraction mapping theorem. Lastly, we check that the fixed
point is the solution of \eqref{short time linearized: itr equ}. Obviously, there holds
\begin{equation*}
v=T(v)=\int_0^t\int_M{H(x,y,t-s)\{Q+u-\tri v+\tri_{\vphi}v\}(y,s)}\om^n(y)ds.
\end{equation*} So we get $v\in C^{2+\a,1}(M\times [0,t_0])$. Then after differentiating on the both sides, we obtain that $v$ satisfies \eqref{short time: linearized equ} on $[0,t_0]\times{M}$  on account of Proposition~\ref{short time linearized: heat kernel} (\textrm{vi}).
As a result, we conclude that
\begin{thm}\label{short time linearized: local sol}
The linearized equations
\eqref{short time: linearized equ}
have a local solution $v\in C^{2+\a,1}(M\times [0,t_0])$ under the normalization condition
\eqref{short time: Q normalization condition}.
\end{thm}
\begin{rem}
In Kru\u{z}kov-Kastro-Lopes \cite{MR0393848},
they proved the following. Suppose $\Om=R^n$,
$f=f_1+f_2$, $f_1\in B([0,T],C^{\b})$,
$f_2\in B([0,T],C^{\a})$, $\a>\b$ and
$u(t,x)\in B([0,T],C^{2,\b})$
is a solution of \eqref{regularity of PCF: linear parabolic};
then $u\in C^{0,2+\b}$ and
\begin{align}\label{regularity of PCF: space reg potential}
\sup_{[0,t]}|u|_{C^{2,\b}}\leq M[\sup_{[0,T]}|f_1|_{C^{2,\b}}
+\int_0^t(t-\tau)^{\frac{\a-\b}{2}-1}\sup_{[0,\tau]}|f_2|_{C^{2,\b}}d\tau]
\end{align}
is satisfied for any $0\leq t\leq T$.
Note that in our linearized equation,
$u-\tri v+\tri_\vphi v\in C^0([0,T],C^{\a})$ holds and
$Q\in C^0([0,T],C^{\gamma})$ for any $\gamma$ satisfying
$1>\gamma\geq\a$ follows from \lemref{short time: Q regular}; so their estimate is sufficient in our proof.
\end{rem}
In the next subsection, we are going to prove an a priori estimate of the solution of the linearized equation to make sure that the short time solution of this linearized equation can be extended to any
fixed time $t$ under the condition \eqref{short time: vphi small} on $\vphi$.
\subsection{A priori estimates of the linearized equation}\label{estimate of linearized}
Suppose $v$ is the solution of the linearized equations
\eqref{short time: linearized equ}
under the normalization condition \eqref{short time: Q normalization condition}. Meanwhile, suppose that $\vphi\in  V_\delta$ (see \eqref{short time: v delta}) satisfies
\begin{equation}\label{short time: app sol}
  \left\{
   \begin{aligned}
\frac{\p\vphi}{\p t}
&=\log\frac{\om^n_\vphi}{\om^n}-\tilde{P},\\
\vphi(0)&=\vphi_0. \\
   \end{aligned}
  \right.
\end{equation}
Here $\tilde{P}$ belongs to $C^0([0,T],C^{2+\a}(M))$.
In this section, we will prove the following extension theorem.
\begin{thm}\label{estimate of linearized: est}
For any $T>0$, the equation \eqref{short time: linearized equ} with \eqref{short time: Q normalization condition} has a solution $v\in C^{2+\a,1}(Q_T)$.
\end{thm}
Since our potential $\vphi$ belongs to
$C^0([0,T],U)$, \lemref{short time: app coe t} implies that there is a
sequence $\vphi_n\in C^{\infty}(Q_T)$ such that
$$\vphi_n\rightarrow\vphi\in C^0([0,T],U).$$ In the following lemmas, since the
constants in all inequalities do not contain any derivatives of
$\vphi$ more than
$\max_{[0,T]}|\p_t\vphi|_{C^{\a}(M)}+\max_{[0,T]}|\vphi|_{C^{2,\a}(M)}$,
we omit the approximation process for convenience.
\begin{lem}\label{estimate of linearized: L^p Kahler}
If $v$ satisfies the following heat equation with the initial data $\vphi_0$
\begin{equation*}
  \left\{
   \begin{aligned}
\frac{\p v}{\p t}-\tri_{\vphi}v
&=u(x,t) \text{ in } Q_T,  \\
v(0)&=w, \\
   \end{aligned}
  \right.
\end{equation*}
then we have
\begin{equation*}
||\tri_\vphi v||^2_{2;Q_T,g_\vphi}\leq
C(||v||^2_{2;Q_T,g_\vphi}+||u||^2_{2;Q_T,g_\vphi}
+||w||^2_{2;M,g_\vphi}+||\tri_\vphi w||^2_{2;M,g_\vphi}),
\end{equation*}
where $C$ depends on $n, Q_T, \lambda_0, \Lambda_0$ and the moduli of continuity of $g^{i\bar j}_\vphi$ on $M$.
\end{lem}
\begin{proof}
Since $\tilde{v}=v-w$ satisfies the following equation with zero initial data
\begin{equation*}
\frac{\p\tilde{v}}{\p t}-\tri_{\vphi}\tilde{v}
=u(x,t)+\tri_{\vphi}w \text{ in } Q_T,
\end{equation*}
by the $L^2$ estimate we obtain
\begin{equation*}
||\tri_\vphi\tilde{v}||^2_{2;Q_T,g_\vphi}\leq
C(||\tilde{v}||^2_{2;Q_T,g_\vphi}
+||u+\tri_{\vphi}w||^2_{2;Q_T,g_\vphi}).
\end{equation*}
\end{proof}
\begin{lem}\label{estimate of linearized: energy estimate on M}
There exists a constant $C$ depending on $\tilde{P}$ such that
\begin{align}
\frac{1}{2}\int_{M}|v|^2\p_t\om_\vphi^n\leq\delta||\nabla_{\vphi}v||
^2_{2;M,g_{\vphi}}+C||v||^2_{2;M,g_\vphi}+\delta^2||\tri_\vphi
v||^2_{2;M,g_\vphi}.
\end{align}
\end{lem}
\begin{proof}
Using \eqref{short time: app sol}, we express the left hand side of the inequality in the statement as
\begin{align*}
&\frac{1}{2}\int_{M}|v|^2\p_t\om_\vphi^n
=\frac{1}{2}\int_{M}|v|^2\tri_\vphi(\p_t\vphi)\om_\vphi^n\\
&=\frac{1}{2}\int_{M}|v|^2\tri_\vphi(\log\frac{\om_\vphi^n}{\om^n}
-\tilde{P})\om_\vphi^n\\
&=\int_M|\nabla_\vphi v|^2\log\frac{\om_\vphi^n}{\om^n}\om_\vphi^n
+\int_Mv\tri_\vphi v\log\frac{\om_\vphi^n}{\om^n}\om_\vphi^n
-\frac{1}{2}\int_{M}|v|^2\tri_\vphi \tilde{P}\om_\vphi^n.
\end{align*}
Therefore the lemma follows from the condition \eqref{short time: vphi small} and the Cauchy-Schwarz inequality.
\end{proof}
\begin{lem}\label{estimate of linearized: energy estimate on Q}
If $v$ is the solution of \eqref{short time: linearized equ}, then we have
\begin{align}
&\frac{1}{2}\int_{Q_T}|v|^2\p_t\om_\vphi^n
\leq\delta||\nabla_{\vphi}v||^2_{2;Q_T,g_\vphi}\nonumber\\
&+C(||v||^2_{2;Q_T,g_\vphi}+||u||^2_{2;Q_T,g_\vphi}
+||w||^2_{2;M,g_\vphi}+||\tri_\vphi w||^2_{2;M,g_\vphi})
\end{align}
where $C$ depends on $n, Q_T, \lambda_0, \Lambda_0, \tilde{P},
|T|_{C^{0}(Q_T)}$ and the moduli of continuity of $g^{i\bar j}_\vphi$.
\end{lem}
\begin{proof}
By using \eqref{short time: app sol}, we calculate
\begin{align*}
&\frac{1}{2}\int_{Q_T}|v|^2\p_t\om_\vphi^n\wedge dt
=\frac{1}{2}\int_{Q_T}|v|^2\tri_\vphi(\p_t\vphi)\om_\vphi^n\wedge dt\\
&=\frac{1}{2}\int_{Q_T}|v|^2
\tri_\vphi(\log\frac{\om^n_\vphi}{\om^n}-\tilde{P})\om_\vphi^n\wedge dt\\
&=\int_{Q_T}|\nabla_\vphi v|^2\log\frac{\om^n_\vphi}{\om^n}\om_\vphi^n\wedge dt+\int_{Q_T}v\tri_\vphi
v\log\frac{\om^n_\vphi}{\om^n}\om_\vphi^n\wedge dt\\
&-\frac{1}{2}\int_{Q_T}|v|^2\tri_\vphi \tilde{P}\om_\vphi^n\wedge dt.
\end{align*}
Using the condition \eqref{short time: vphi small} and the Cauchy-Schwarz inequality we get
\begin{align*}
\frac{1}{2}\int_{Q_T}|v|^2\p_t\om_\vphi^n\wedge dt
\leq\delta||\nabla_{\vphi}v||^2_{2;Q_T,g_\vphi}
+C||v||^2_{2;Q_T,g_\vphi}+\delta^2||\tri_\vphi
v||^2_{2;Q_T,g_\vphi}.
\end{align*}
By applying \lemref{estimate of linearized: L^p Kahler} to \eqref{short time: linearized equ}, we obtain that the right-hand side can be bounded by terms
\begin{align*}
&\delta||\nabla_{\vphi}v||^2_{2;Q_T,g_\vphi}
+C||v||^2_{2;Q_T,g_\vphi}\\
&+\delta^2C(||v||^2_{2;Q_T,g_\vphi}
+||u+Q||^2_{2;Q_T,g_\vphi}
+||w||^2_{2;M,g_\vphi}+||\tri_\vphi w||^2_{2;M,g_\vphi}).
\end{align*}
Furthermore, \lemref{short time: Lp estimate of Q} gives the bound of the above
\begin{align*}
\delta||\nabla_{\vphi}v||^2_{2;Q_T,g_\vphi}
+C(||v||^2_{2;Q_T,g_\vphi}+||u||^2_{2;Q_T,g_\vphi}
+||w||^2_{2;M,g_\vphi}+||\tri_\vphi w||^2_{2;M,g_\vphi}).
\end{align*} Therefore the lemma follows.
\end{proof}
\begin{lem}
If $v$ is the solution of \eqref{short time: linearized equ}, then we have
\begin{align}\label{estimate of linearized: energy estimate of v}
\sup_{0\leq t\leq T}||v||_{2;M,g_\vphi}&\leq
C(||u||_{L^{2}(Q_T,g_\vphi)}+||w||_{2;M,g_\vphi}
+||\tri_\vphi w||_{2;M,g_\vphi}),
\end{align}
where $C$ depends on $n, Q_T, \l_0, \L_0, \tilde{P},
|T|_{C^{0}(Q_T)}$ and the moduli of continuity of $g^{i\bar j}_\vphi$.
\end{lem}
\begin{proof}
By using \eqref{short time: linearized equ} we calculate
\begin{align*}
&\p_t\frac{1}{2}\int_M|v|^2\om_\vphi^n
=\int_Mv\p_tv\om_\vphi^n
+\frac{1}{2}\int_M|v|^2\tri_\vphi(\p_t\vphi)\om_\vphi^n\\
&=\int_Mv(\tri_{\vphi}v
+Q+u)\om_\vphi^n
+\frac{1}{2}\int_M|v|^2\tri_\vphi(\p_t\vphi)\om_\vphi^n.
\end{align*}
Applying \lemref{short time: L2 estimate of Q} and \lemref{estimate of linearized: energy estimate on M}, we get
\begin{align*}
\p_t\frac{1}{2}\int_M|v|^2\om_\vphi^n
&\leq-||\nabla_{\vphi}v||^2_{M,g_{\vphi}}
+C||v||^2_{2;M,g_\vphi}+||u||^2_{2;M,g_\vphi}\\
&+\delta||\nabla_{\vphi}v||^2_{M,g_{\vphi}}
+C||v||^2_{2;M,g_\vphi}+\delta^2||\tri_\vphi
v||^2_{2;M,g_\vphi}\\
&\leq
C(t)||v||^2_{2;M,g_\vphi}+||u||^2_{2;M,g_\vphi}
+\delta^2||\tri_\vphi v||^2_{2;M,g_\vphi}.
\end{align*}
Then the Gronwall's inequality implies
\begin{align*}
\frac{1}{2}||v||^2_{2;M,g_\vphi}
&\leq
e^{\int_0^tC(s)ds}[\frac{1}{2}||v(0)||^2_{2;M,g_\vphi}
+\int_0^t||u||^2_{2;M,g_\vphi}
+\delta^2||\tri_\vphi v||^2_{2;M,g_\vphi}ds]\\
&\leq C[\frac{1}{2}||w||^2_{2;M,g_\vphi}+||u||^2_{2;Q_T,g_\vphi}
+\delta^2||\tri_\vphi v||^2_{2;Q_T,g_\vphi}]\\
&\leq C[\frac{1}{2}||w||^2_{2;M,g_\vphi}+||u||^2_{2;Q_T,g_\vphi}\\
&+\delta^2C(||v||^2_{2;Q_T,g_\vphi}+||u||^2_{2;Q_T,g_\vphi}
+||w||^2_{2;M,g_\vphi}+||\tri_\vphi w||^2_{2;M,g_\vphi}].
\end{align*}
The last inequality follows from
\lemref{estimate of linearized: L^p Kahler}.
Therefore \eqref{estimate of linearized: energy estimate of v} holds, if we choose
$\delta$ small enough so that $\delta^2C\leq\frac{1}{4}$.
\end{proof}
\begin{lem}
If $v$ is the solution of \eqref{short time: linearized equ}, then we have
\begin{align}\label{estimate of linearized: energy estimate of nabla v}
||\nabla_{\vphi}v||^2_{2;Q_T,g_\vphi}
\leq C(||u||^2_{2;Q_T,g_\vphi}+||w||^2_{2;M,g_\vphi}
+||\tri_\vphi w||^2_{2;M,g_\vphi})
\end{align}
where $C$ depends on $n, Q_T, \l_0, \L_0, \tilde{P},
|T|_{C^{0}(Q_T)}$ and the moduli of continuity of $g^{i\bar j}_\vphi$.
\end{lem}
\begin{proof}
Using \eqref{short time: linearized equ} we have
\begin{align*}
||\nabla_\vphi v||^2_{2;Q_T}&=-\int_0^T\int_Mv\tri_\vphi v\om_\vphi^n\wedge dt=\int_{Q_T}\{-v\frac{\p
v}{\p t}+vQ\}\om_\vphi^n\wedge dt
+\int_{Q_T}vu\om_\vphi^n\wedge dt\\
&=-\frac{1}{2}[\int_{Q_T}\p_t(|v|^2\om_\vphi^n)\wedge dt
-\int_{Q_T}|v|^2\p_t\om_\vphi^n\wedge dt]
+\int_{Q_T}v(Q+u)\om_\vphi^n\wedge dt.
\end{align*}
Then applying \lemref{short time: Lp estimate of Q} and \lemref{estimate of linearized: energy estimate on Q} we have
\begin{align*}
||\nabla_\vphi v||^2_{2;Q_T}&\leq\frac{1}{2}||w||^2_{2;M,g_\vphi}+\delta||\nabla_{\vphi}v||^2_{2;Q_T,g_\vphi}
+C(||v||^2_{2;Q_T,g_\vphi}+||u||^2_{2;Q_T,g_\vphi}\\
&+||w||^2_{2;M,g_\vphi}+||\tri_\vphi w||^2_{2;M,g_\vphi})
+C(||v||^2_{2;Q_T,g_\vphi}+||u||^2_{2;Q_T,g_\vphi}).
\end{align*}
Consequently, we get to the following estimate
\begin{align*}
||\nabla_{\vphi}v||^2_{2;Q_T,g_\vphi}
\leq C(||u||^2_{2;Q_T,g_\vphi}+||w||^2_{2;M,g_\vphi}
+||\tri_\vphi w||^2_{2;M,g_\vphi}),
\end{align*} provided that
$\delta\leq\frac{1}{2}$.
\end{proof}
\begin{prop}(Energy inequality)\label{estimate of linearized: energy estimate}If $v$ is the solution of \eqref{short time: linearized equ}, then we have that $v$ is uniformly bounded in $V_2(Q_T)$;
\begin{align*}
||v||_{V_2(Q_T)}=\sup_{0\leq t\leq T}||v||_{2;M}+||\nabla
v||_{2;Q_T}\leq C[||w||_{2;M}+||\tri_\vphi w||_{2;M}+||u||_{2;Q_T}],
\end{align*}
where $C$ depends on $n, Q_T, \l_0, \L_0, \tilde{P},
|T|_{C^{0}(Q_T)}$ and the moduli of continuity of $g^{i\bar j}_\vphi$.
\end{prop}
\begin{proof}
One can triangulate $g$
and $g_\vphi$ at the same time in the normal coordinates such that
$g_{i\bar j}=\delta_{ij}$ and $g_{\vphi i\bar j}=\delta_{ij}+\vphi_{i\bar i}$.
Then we get
\begin{equation*}
\frac{1}{\L_0^2}w_{i\bar
k}w_{k\bar i}\leq g_\vphi^{i\bar j}g_\vphi^{k\bar l}w_{i\bar l}w_{k\bar
j}=\frac{1}{1+\vphi_{i\bar i}}\frac{1}{1+\vphi_{k\bar
k}}w_{i\bar k}w_{k\bar i}\leq\frac{1}{\l_0^2}w_{i\bar
k}w_{k\bar i},
\end{equation*} since $\l g\leq g_{\vphi}\leq\L g$.
Accordingly, \begin{equation*}
||\tri_\vphi w||^2_{2;M,g_\vphi}
=\int_Mg_\vphi^{i\bar j}g_\vphi^{k\bar l}w_{i\bar l}w_{k\bar j}\om_\vphi^n
\end{equation*} is equivalent to $||\tri_\vphi w||^2_{2;M,g}$.
Analogously, $||\nabla v||_{2;M,g_\vphi}$ and $||v||_{2;M,g_\vphi}$ are equivalent to $||\nabla v||_{2;M,g}$ and $||v||_{2;M,g}$ respectively.
As a result we combine \eqref{estimate of linearized: energy estimate of v}
and \eqref{estimate of linearized: energy estimate of nabla v}
with the conclusion above to
obtain the energy inequality.
\end{proof}
We introduce the following imbedding theorem (\cite{MR0241822} Page 77).
Let $m=2n$ be the real dimension.
\begin{thm}(Lady\u{z}enskaja, Solonnikov and Ural'ceva \cite{MR0241822})\label{estimate of linearized: sobolev v2}
If $v\in V_2(Q_T)$, one has the imbedding estimate $$|v|_{W^{q,r}(Q_T)}\leq C|v|_{V_2(Q_T)}$$ for
$$C=2\b(r,n,q)+T^{\frac{1}{r}}V^{-\frac{1}{2}+\frac{1}{q}}$$ and $q,r$ satisfying the relation
$\frac{1}{r}+\frac{m}{2q}=\frac{m}{4}$ with
\begin{equation*}
\begin{cases}
r\in[2,\infty],q\in[2,\frac{2m}{m-2}] \text{ for } m>2;\\
r\in[2,\infty],q\in[2,\infty] \text{ for } m=2.\\
\end{cases}
\end{equation*}
\end{thm}
The Sobolev imbedding theorem states:
\begin{thm}(Lady\u{z}enskaja, Solonnikov and Ural'ceva \cite{MR0241822})\label{estimate of linearized: sobolev parabolic}
Suppose that $v\in W_p^{2k,k}(Q_T),p\geq1,m\geq2$ and $0\leq
r+2s=\mu<2k$.\\
If $(2k-\mu)p<m+2$, then $D_t^sD_x^rv\in L^{q}(Q_T)$ for
$q=\frac{(m+2)p}{m+2-(2k-\mu)p}$;\\
if $(2k-\mu)p>m+2$, then $D_t^sD_x^rv\in C^{\a}(Q_T)$ for
$\a=(2k-\mu)-\frac{m+2}{p}$.
\end{thm}
These imbedding theorems and the energy inequality imply the following estimate by the standard bootstrap method.
\begin{lem}\label{estimate of linearized: v 1,a}
There exists a positive constant $C$ such that
$$|v|_{C^{1+\a,\frac{1}{2}+\frac{\a}{2}}(Q_T)}\leq
C(||w||_{W^{2,p}(M)}+||u||_{p;Q_T}).$$
\end{lem}
\begin{proof}
Owing to the energy inequality,
Proposition~\ref{estimate of linearized: energy estimate},
we have $v\in
V_2(Q_T)$. Then the imbedding
theorem, \thmref{estimate of linearized: sobolev v2}, implies a
uniform $L^{p}(Q_T)$ bound for $p=\frac{2(m+2)}{m}$ i.e.~$$
||v||_{L^{p}(Q_T)}\leq C||v||_{V_2(Q_T)}.$$ According to
~\lemref{short time: Lp estimate of Q}
we obtain that $Q$ has uniform $L^{p}(Q_T)$ bound,
i.e. $$||Q||_{L^{p}(Q_T)}\leq C||v||_{L^{p}(Q_T)}.$$ Then the
parabolic $L^p$ theory tells us $v$ has uniform $W_{p}^{2,1}(Q_T)$
bound; i.e.
\begin{align*}
||v||_{W_{p}^{2,1}(Q_T)}&\leq
C(||v||_{L^{p}(Q_T)}+||Q||_{L^{p}(Q_T)}+||u||_{L^{p}(Q_T)}
+||w||_{W^{2,p}(M)})\\
&\leq C(||v||_{L^{p}(Q_T)}+||u||_{L^{p}(Q_T)}+||w||_{W^{2,p}(M)}).
\end{align*}
The Sobolev imbedding theorem, \thmref{estimate of linearized: sobolev parabolic}, implies $$v\in L^{p_1}(Q_T)$$ for
$p_1=\frac{(m+2)p}{m+2-2p}=\frac{2(m+2)}{m-4}>p$.

Analogously, we get $Q\in L^{p_1}(Q_T)$ and
$v\in W_{p_1}^{2,1}(Q_T)$. Repeating the similar
process, we derive $p_k=\frac{2(m+2)}{m-4k}>m+2$ for some step $k$.
Then according to the Sobolev imbedding
theorem, \thmref{estimate of linearized: sobolev parabolic}, we have $$v\in C^{1+\a,\frac{1+\a}{2}}(Q_T).$$ Thus the lemma follows.
\end{proof}

\begin{prop}\label{estimate of linearized: v 2+a,1}
There exists a positive constant $C$ such that
$$|v|_{C^{2+\a,1}(Q_T)}\leq
C(|w|_{C^{2,\a}(M)}+\sup_{[0,T]}|u|_{C^{\a}(M)}).$$
\end{prop}
\begin{proof}
Because of \lemref{short time: app coe t} one can select a sequence $\vphi_n\in C^\infty(Q_T)$ and $T_{ni\bar{j}}\in
C^\infty(Q_T)$ such that $\vphi_{n}\rightarrow \vphi\in
C^0([0,T],C^{2,\a})$ and $T_{ni\bar{j}}\rightarrow T_{i\bar{j}}\in
C^0([0,T],C^{\a})$. We omit the standard approximation argument in the sequel of the proof.
Since $M$ is compact,
$M$ can be covered by finite number of balls $\{B_p(2r)\}$
with radii $2r$. Let $0\leq\eta_p\leq1$
be a smooth partition of unity
subordinate to the covering $\{B_p(2r)\}$. Since $T_{i\bar j}$ is a harmonic tensor, using \eqref{short time: Q} we have
\begin{align*}
\tri_\vphi Q=[v^{i}T_{i\bar{j}}]^{\bar{j}}.
\end{align*}
Now on each ball $B(2r)\in \{B_p(2r)\}$,
we calculate
\begin{align}\label{short time: eta Q}
\tri_\vphi (\eta Q)=[\eta v^{i}T_{i\bar{j}}]^{\bar{j}}-\eta^{\bar{j}} v^{i}T_{i\bar{j}}
+\tri_\vphi \eta  Q + 2\nabla_\vphi \eta  \nabla_\vphi Q.
\end{align}
Then the Schauder estimate for the first
derivatives in \cite{MR1814364} (Corollary 8.35) provides us the following estimate
\begin{align*}
\sup_{[0,T]}|\eta  Q|_{C^{1,\a}(B(r))}
&\leq C(\sup_{[0,T]}|Q|_{C^{0}(B(2r))}
+\sup_{[0,T]}|v^{i}T_{i\bar{j}}|_{C^{\a}(B(2r))}\\
&+\sup_{[0,T]}|-\eta^{\bar{j}} v^{i}T_{i\bar{j}}
+\tri_\vphi \eta  Q
+2\nabla_\vphi\eta \nabla_\vphi Q|_{C^{0}(B(2r))}.
\end{align*}
Combining these estimates on each covering and using the interpolation inequality for the H\"older space,
we obtain
\begin{align}\label{short time: Q 1 a}
\sup_{[0,T]}|Q|_{C^{1,\a}(M)}
\leq C(\sup_{[0,T]}|Q|_{0;M}+\sup_{[0,T]}|v|_{C^{1,\a}(M)}).
\end{align}
Since $\eta v$ satisfies \eqref{short time: eta Q} on the ball $B(2r)$,
the $L^\infty$ bound of $Q$ follows
from Theorem 8.16 in \cite{MR1814364}, for some $q>n$
\begin{align*}
\sup_{B_{2r}}|\eta Q|
&\leq C(||\eta v^{i}T_{i\bar{j}}||_{L^q(B_{2r})}
+||-\eta^{\bar{j}}v^{i}T_{i\bar{j}}
+\tri_\vphi\eta  Q
+2\nabla_\vphi\eta\nabla_\vphi Q||_{L^{\frac{q}{2}}(B_{2r})})\\
&\leq C(|v|_{C^1}(B_{2r})+||Q||_{L^{\frac{q}{2}}(B_{2r})}
+||\nabla Q||_{L^{\frac{q}{2}}(B_{2r})}).
\end{align*}
Note that the interpolation inequality
for the Sobolev space implies
$$||\nabla Q||_{L^{\frac{q}{2}}(B_{2r})}
\leq \eps||Q||_{W^{2,\frac{q}{2}}(B_{2r})}
+C(\eps)||Q||_{L^{\frac{q}{2}}(B_{2r})};$$
moreover, the $L^p$ estimate of \eqref{short time: Q} gives that
$$||Q||_{W^{2,\frac{q}{2}}(B_{2r})}\leq C(||Q||_{L^\frac{q}{2}(M)}+|v|_{C^2}(M)).$$
Combining the above estimates, we infer that
\begin{align}\label{estimate of linearized: Q C0}
\max_{M}|Q|\leq\sum_q\max_{B_q(2r)}|\eta_q Q|
\leq C(|v|_{C^1(M)}+||Q||_{L^{\frac{q}{2}}(M)}
+\eps|v|_{C^{2}(M)}).
\end{align}
Plugging \lemref{estimate of linearized: v 1,a} and
\eqref{short time: Q 1 a} in
\thmref{regularity of PCF: space reg} we obtain
\begin{align*}
&\sup_{[0,T]}|\p_tv|_{C^{\a}(M)}
+\sup_{[0,T]}|v|_{C^{2,\a}(M)}\\
&\leq
C(\sup_{[0,T]}|Q|_{C^{\a}(M)}+|w|_{C^{2,\a}(M)}
+\sup_{[0,T]}|u|_{C^{\a}(M)})\\
&\leq
C(\sup_{[0,T]}|Q|_{0;M}+\sup_{[0,T]}|v|_{C^{1,\a}(M)}+|w|_{C^{2,\a}(M)}
+\sup_{[0,T]}|u|_{C^{\a}(M)})\\
&\leq C(\sup_{[0,T]}|Q|_{0;M}+
|w|_{C^{2,\a}(M)}+\sup_{[0,T]}|u|_{C^{\a}(M)}).
\end{align*}
Then \eqref{estimate of linearized: Q C0} gives the bound of the right-hand side
\begin{align*}
C(\eps\sup_{[0,T]}|v|_{C^{2}(M)}+
|w|_{C^{2,\a}(M)}+\sup_{[0,T]}|u|_{C^{\a}(M)})
\end{align*}
which implies the estimate of $v$
\begin{align}\label{estimate of linearized: est v}
\sup_{[0,T]}|\p_tv|_{C^{\a}(M)}
+\sup_{[0,T]}|v|_{C^{2,\a}(M)}
\leq C(|w|_{C^{2,\a}(M)}+\sup_{[0,T]}|u|_{C^{\a}(M)})
\end{align}
provided $C\eps<\frac{1}{2}$.
\end{proof}
\begin{proof}(proof of \thmref{estimate of linearized: est})
Now we complete our proof of \thmref{estimate of linearized: est} by contraction. If the time of existence cannot be extended beyond some $t_0<T$, then the $C^{2,\a}$ norm must blow up at time $t_0$.
This contradicts
the a priori $C^{2,\a}$ estimates, Proposition~\ref{estimate of linearized: v 2+a,1}, we derived.
\end{proof}


Finally we consider the time regularity of the solution.
\begin{lem}\label{estimate of linearized: con in t}
Let
$w\in C^{2,\a}$ and $u\in C^0([0,T],C^{2,\a})$. Then the solution of \eqref{short time: linearized equ} stays in
$$C^0([0,T],C^{2,\a})\cap C^1([0,T],C^{\a}).$$
\end{lem}
\begin{proof}
Choose two sequences $w_n\in C^{\a+\eps}(M)$ and $u_n \in C^0([0,T],C^{\a+\eps}(M))$ for $0\leq\eps<1-\a$ such that
\begin{align*}
&w_{n}\rightarrow w \in C^{\a}(M)\text{ and } u_{n}\rightarrow u \in C^0([0,T],C^{\a}(M)).
\end{align*}
According to \thmref{short time linearized: local sol}
and the estimate \eqref{estimate of linearized: est}, for each $n$
there exits a solution $v_n\in C^{2+\a+\eps,1}(M\times[0,T])$
of the linearized equation
 \begin{equation*}
  \left\{
   \begin{aligned}
\frac{\p v}{\p t}-\tri_{\vphi}v-Q(v)
&=u_n \text{ in } Q_T,\\
\tri_{\vphi} Q(v)&=v^
{i\bar{j}}T_{i\bar{j}} \text{ in } Q_T,\\
v(0)&=w_n. \\
   \end{aligned}
  \right.
\end{equation*} It is directly verified that
$v_n\in C^0([0,T],C^{2,\a})\cap C^1([0,T],C^{\a})$.
Moreover, \eqref{estimate of linearized: est} implies
for any $t,s\in [0,T]$,
\begin{align*}
\sup_{[0,T]}|v_m-v_n|_{C^{2,\a}(M)}
\leq C(|w_m-w_n|_{C^{2,\a}(M)}
+\sup_{[0,T]}|u_m-u_n|_{C^{\a}(M)}).
\end{align*}
From this we immediately
conclude that $v_n$ is a Cauchy sequence
in $C^0([0,T],C^{2,\a})$. Hence we obtain $v\in C^0([0,T],C^{2,\a})$ and the lemma follows from \eqref{short time: linearized equ}.
\end{proof}

\subsection{Regularity of the pseudo-Calabi flow}\label{regularity of PCF}
We start with recalling an interior regularity theorem
for the linear equation
in \cite{MR0241822} and \cite{MR1465184}.
Let $Q=M\times (T_1,T_2)$
and $Q'=M\times(T_3,T_4)\subset Q$,
where $0<T_1<T_3<T_4<T_2$.

\begin{thm}($L^p$ interior estimate \cite{MR0241822}\cite{MR1465184})\label{regularity of PCF: Lp estimate of standard equ}
 Let $u\in W_{loc}^{k+2,p}(Q)$ for some $p\geq 2$
 be the solution of the linear
 equation \eqref{regularity of PCF: linear parabolic}
that satisfies the following conditions
$$a^{ij}\xi_i\xi_j\geq\l|\xi|^2>0 \text{ for any }\xi
\in R^n\backslash\{0\},$$
$$|a^{ij}|_{W^{k,\infty}(Q)}+|b^i|_{W^{k,\infty}(Q)}
+|c|_{W^{k,\infty}(Q)}\leq M$$ and
$$|a^{ij}(X)-a^{ij}(Y)|=\om(\frac{|X-Y|}{d(X,Y)})$$
for some positive constants $\l$ and $M$ and a positive,
continuous, increasing
function $\om$ with $\om(0)= 0$.
Then for any $T_3-T_1>\eps$ we have
$$|u|_{W^{k+2,p}(Q')}
\leq C(k+2,n,p,\l,M,\eps)(|u|_{L^p(Q)}+|f|_{W^{k,p}(Q)}).$$
\end{thm}
Consider the linear parabolic equation of divergence form
\begin{align}\label{regularity of PCF: linear parabolic div}
u_t=g^{p\bar q}u_{p\bar q}+b^iu_i+cu+f.
\end{align}
Here $g_{p\bar q}$ is a K\"ahler metric and $g_{ij}$ for $1\leq i,j\leq 2n$ is the corresponding $J$-invariant Riemannian metric.
Similarly to \thmref{regularity of PCF: space reg}, we can generalize the first part of \thmref{regularity of PCF: space reg Rn} on a Riemannian manifold $M$.
\begin{thm}\label{regularity of PCF: weak space reg}
Let $f\in UC(M\times[0,t_0])\cap C^{\a,0}(M\times[0,t_0])$
and $u_0\in C^{0}(M)$.
If the coefficients $g^{ij}$, $b^i$ and $c$ belong to $C^{\a}(M)$
with $0<\a< 1$, and $g^{i\bar j}$ satisfies
 $\L|\xi|^2\geq g^{ij}\xi_i\xi_j\geq\l|\xi|^2>0
\text{ for any }\xi\in R^n\backslash\{0\}$. If $u_0\in C^0(M)$,
then \eqref{regularity of PCF: linear parabolic div} has a unique solution which belongs to $C([0,T]\times\Om)$.  For every $\eps\in (0,t_0]$ there is a constant $C$ such that
\begin{align*}
||u||_{C^{2+\a,1}(M\times[\eps,t_0])}
\leq C(\eps)(|u_0|_{C^{0}(M)}
+|f|_{C^{\a,0}(M\times[0,t_0])}).
\end{align*}
\end{thm}
Denote $Q_{T_\eps}=M\times[\eps,T]$
where $0<\eps<T$;
we have the following result.
\begin{prop}\label{regularity of PCF: classical solution reg}
The $C^{2,1}(Q_{T_\eps})\cap C^0([\eps,T],C^{2}(M))$ solution
of the pseudo-Calabi flow
belongs to $$C^0([\eps,T],C^{2+\a}(M))\cap
C^1([\eps,T],C^{\a}(M)).$$
\end{prop}
\begin{proof}
Since we assume $\vphi\in C^{2}(Q_{T_\epsilon})$, applying the $L^p$ estimate to the equation \eqref{notations: P}
with the normalization condition \eqref{notations: P normalization condition},
we have for any
$1<p<\infty$ $$P\in C^0([\eps,T],W^{2,p}(M)).$$
Then the embedding theorem implies that $P$ belongs to
$C^0([\eps,T],C^{1,\a}(M))$.

In order to obtain a higher regularity, we
consider the equation
in a coordinate chart $(\mathcal{U},x_k)$
for $1\leq k\leq 2n$,
and select
a direction $e_k=\frac{\p}{\p{x^k}}$
in the tangent space
$T_R{M}$. We define the difference quotient
of $\vphi$ at $x$ in the direction $e_k$ as
\begin{align*}
\vphi_\rho=\frac{\vphi(x+\rho e_k,t)-\vphi(x,t)}{\rho}.
\end{align*}
Taking the difference quotients
of both sides of the equation \eqref{short time: pam of PCF},
we obtain
\begin{align}\label{regularity of PCF: x deri of PCF}
\frac{\p\vphi_\rho}{\p
t}&=\frac{h(x+\rho e_k,t)-h(x,t)}{\rho}+\frac{P(x+\rho e_k,t)-P(x,t)}{\rho}\\
&=\frac{\log\frac{\om_{\vphi}^n}{\om^n}(x+\rho e_k,t)-\log\frac{\om_{\vphi}^n}{\om^n}(x,t)}{\rho}+\frac{P(x+\rho e_k,t)-P(x,t)}{\rho}\nonumber\\
&=\frac{1}{\rho}\log\frac{\om_{\vphi}^n(x+\rho
e_k,t)}{\om_{\vphi}^n(x,t)}-\frac{1}{\rho}\log\frac{\om^n(x+\rho
e_k,t)}{\om^n(x,t)}+\p_kP(x+ \vartheta e_k,t).\nonumber
\end{align}
We can see that the second term is smooth
and the third term belongs to
$C^0(Q_{T_\epsilon})$.

Now we deal with the first term. In the local
coordinate chart $(\mathcal{U},x_k)$ for $1\leq k\leq 2n$,
we can choose some smooth real valued function $\psi$
such that $\om=\frac{\sqrt{-1}}{2}\p\bar\p \psi$. Then the first term can be expressed as
\begin{align*}
\frac{1}{\rho}\log\frac{\om_{\vphi}^n(x+\rho
e_k,t)}{\om_{\vphi}^n(x,t)}
&=\frac{1}{\rho}\log\frac{\det(g_{i\bar{j}}+\vphi_{i\bar{j}})(x+\rho
e_k,t)}{\det(g_{i\bar{j}}+\vphi_{i\bar{j}})(x,t)}\\
&=\frac{1}{\rho}
\int_0^1\frac{\p}{d\theta}\log\det(g_{\theta i\bar{j}})d\theta\\
&=\int_0^1g_{\vphi\theta}^{i\bar{j}}d\theta
(\vphi_{\rho}+\psi_{\rho})_{i\bar{j}},
\end{align*}
where $g_{\vphi\theta}=\theta g_{\vphi}(x+\rho
e_k,t)+(1-\theta)g_{\vphi}(x,t)$
and $\psi_\rho=\frac{\psi(x+\rho
e_k,t)-\psi(x,t)}{\rho}$. Since
$$a^{ij}=\int_0^1g_{\vphi\theta}^{i\bar{j}}d\theta\in C^0(Q_{T_\epsilon})$$
is uniformly elliptic and $\p_kP\in C^0(Q_{T_\epsilon})$,
Proposition~\ref{regularity of PCF: Lp estimate of standard equ}
implies $$\vphi_\rho\in W^{2,p}(Q_{T_\eps})$$ for any $p>1$.
After letting $\rho\rightarrow0$, we get
$\frac{\p\vphi}{\p x^k}$ in $W^{2,p}(Q_{T_\eps})$. The imbedding theorem further implies $\frac{\p\vphi}{\p x^k}\in C^{1,\a}(Q_{T_\eps})$. So we have $$\vphi\in C^0([\eps,T],C^{2+\a}(M)).$$
Since $P\in C^0([\eps,T],C^{1,\a}(M))$, by using the equation \eqref{short time: pam of PCF} we have
$$\p_t\vphi\in C^0([\eps,T],C^{\a}(M)).$$
Hence the proposition follows.
\end{proof}
\begin{prop}\label{regularity of PCF: higher reg}
The $C^0([\epsilon,T],C^{2+\a}(M))\cap
C^1([\epsilon,T],C^{\a}(M))$
solution of the pseudo-Calabi flow equation belongs to $C^{\infty}(M\times[2\epsilon,T])$.
\end{prop}
\begin{proof}
\lemref{short time: P regular} implies $P\in
C^0([\epsilon,T],C^{2+\a}(M))$. Then we have $$\p_k P\in
C^0([\epsilon,T],C^{\a}(M)) \text{ and }a^{ij}\in
C^0([\epsilon,T],C^{\a}(M)).$$
We consider the equation by freezing the coefficient of \eqref{regularity of PCF: x deri of PCF},
\begin{align}\label{regularity of PCF: x deri of PCF frz}
\frac{\p\vphi_\rho}{\p
t}&=\tri\vphi_\rho
+(\int_0^1g_{\vphi\theta}^{i\bar{j}}d\theta-g^{i\bar j})\vphi_{\rho_{i\bar{j}}}
+\int_0^1g_{\vphi\theta}^{i\bar{j}}d\theta\psi_{\rho_{i\bar{j}}}\\
&-\frac{1}{\rho}\log\frac{\om^n(x+\rho
e_k,t)}{\om^n(x,t)}+\p_kP(x+ \vartheta e_k,t).\nonumber
\end{align}
Therefore, we obtain
$\p_k\vphi\in C^0([\epsilon,T],C^{2,\a}(M))$ by
\thmref{regularity of PCF: classical solution reg}
and the condition \eqref{short time: vphi small}
in each coordinate chart $(\mathcal{U},x_k)$.
Moreover, by using \eqref{regularity of PCF: x deri of PCF}
we deduce that $$\p_t\vphi\in
C^0([\epsilon,T],C^{1,\a}(M)).$$
Repeating the same process again
and again,
we have $$\vphi\in
C^0([\eps\sum^q_{p=0}\frac{1}{2^p},T],C^{q+2,\a}(M))$$
for any $q=1,2,\cdots$.
Then taking the derivative of \eqref{short time: pam of PCF}
with respect to variable $t$,
we have $\vphi\in
C^2([2\epsilon,T],C^{\infty}(M))$.
By iteration of this procedure, the proposition follows.
\end{proof}
Since $\eps$ can be arbitrarily small,
we have the following theorem.
\begin{thm}\label{regularity of PCF: reg of PCF}
The $C^0([0,T],C^{2+\a}(M))\cap C^1([0,T],C^{\a}(M))$
solution of the pseudo-Calabi flow equation in fact belongs to
$C^{\infty}(M\times(0,T])$.
\end{thm}
\begin{rem}
The proof still holds if
one
uses the estimate \eqref{regularity of PCF: space reg potential} instead of \thmref{regularity of PCF: space reg Rn}.
\end{rem}
\section{Long time existence of the pseudo-Calabi flow}\label{long time}
In this section we shall use $C_i$ for $i=1,2,\cdots$ to distinguish different generic constants.
\begin{thm}\label{long time: main}
Let $T$ be a finite time.
If $g_{\vphi}(t)$ is a solution of the pseudo-Calabi flow with
$Ric_{\vphi}(t)$ uniformly bounded for all time $t\in[0,T]$, then the pseudo-Calabi flow
can be extended past the time $T$.
\end{thm}
\begin{proof}
Because of the short time existence theorem,
we know that the pseudo-Calabi flow can be restarted at time $T$ if the solution then belongs to $C^{2,\a}(M)$. \thmref{long time: main} will thus follow from the a priori estimates that we shall present in Proposition~\ref{long time: 2+a}.
\end{proof}

\begin{lem}\label{long time: h above}
Suppose that along the pseudo-Calabi flow, there holds
$\int_0^tS_{\vphi}(t)dt\geq-C_1$ for some constant $C_1$. Then there exists a constat $C_2$ depending on $C_1$, $T$ and $\sup_M\log\frac{\om^n_{\vphi_0}}{\om^n}$ such that
$$\sup_Mh\leq C_2.$$
\end{lem}
\begin{proof}
Since the $t$ integral of the scalar curvature
is uniformly bounded below, and
\begin{align*}
\p_t
h=\tri_\vphi\p_t\vphi=-S_\vphi+\ul S,
\end{align*}
then $\sup_Mh\leq C_1+\sup_Mh(0) $.
\end{proof}
Denote the average of $\vphi$
with respect to $\om$
by $\underline{\vphi}=\int_M\vphi\om^n$.
\begin{lem}\label{long time: sup varphi}
There exists a constant $C_3$ such that
\begin{align*}
\sup_{M}(\vphi-\underline{\vphi})\leq C_3.
\end{align*}
\end{lem}
\begin{proof}
Since $ \tri\vphi+n>0, \forall t\geq0$, using the Green
representation we have
\begin{align*}
\vphi(x)-\underline{\vphi}&=
-\frac{1}{V}\int_M\tri\vphi(y)G(x,y)\om^n(y)\\
&\leq n\frac{1}{V}\int_MG(x,y)\om^n(y).
\end{align*}
Since $0\leq G(x,y)\leq \frac{C_4}{d(x,y)^{2n-2}}$, we have
$$\sup_{M\times[0,T]}(\vphi-\underline{\vphi})\leq
C_3=n\frac{1}{V}\int_MG(x,y)\om^n(y).$$
\end{proof}
The lower bound of the normalized potential is obtained by Yau's $C^0$ estimate.
\begin{thm}(Yau \cite{MR480350}) The lower bound of $\vphi-\underline{\vphi}$
is controlled by the upper bound of $h$,
\begin{align}\label{long time: inf varphi}
\inf_{M}(\vphi-\underline{\vphi})\geq-C_5e^{C_6\sup_{M}h}.
\end{align} Here $C_5$ and $C_6$ depend only on $\om$
and $\sup_M(\vphi-\underline{\vphi})$.
\end{thm}

\begin{thm}(Chen-Tian \cite{MR1893004})
Suppose $Ric_{\vphi}$ is bounded from below, then
\begin{align}\label{long time: inf h}
\inf_M{h}&\geq-4C_7e^{2(1+\int_Mh\om_\vphi^n)}.
\end{align}
\end{thm}
Combining \lemref{long time: h above} and \eqref{long time: inf h}, we get the uniform bound of $h$:
\begin{align}\label{long time: h}
\sup_M|h|\leq C_2+4C_7e^{2(1+C_2V)}.
\end{align}
We deduce the lower bound of the average of $\vphi$
by the normalization condition.
\begin{lem}The following estimate holds
\[
\inf_{[0,T]}\underline{\vphi}\geq\underline\vphi(0)+T\inf_Mh.
\]
\end{lem}
\begin{proof}
Due to the normalization condition $\int_Me^P\om^n=Vol$, we apply the Jensen's inequality to get $$\int_MP\om^n\leq0.$$
So we obtain
\[\frac{\p}{\p t}\underline\vphi
=\frac{1}{V}\int_M\frac{\p\vphi}{\p t}\om^n
=\frac{1}{V}\int_M(h-P)\om^n\geq\inf_Mh.
\]
Similarly, we get $\int_Mh\om^n\leq0$ since $\int_Me^h\om^n=V$.
Therefore the lemma follows from $\inf_Mh\leq0$.
\end{proof}
Plugging the above lemma and \lemref{long time: sup varphi} in \eqref{long time: inf varphi},
we get
\begin{align*}
\inf_{M}\vphi\geq-C_5e^{C_6\sup_{M}h}+\underline\vphi(0)+T\inf_Mh.
\end{align*}
Then by \eqref{long time: h} we have that $\vphi$ is uniformly bounded below. On the other hand \eqref{long time: inf varphi} further implies that for a fixed point $p\in M$
\begin{align*}
\underline{\vphi}\leq\inf_{M}\vphi+C_5e^{C_6\sup_{M}h}
\leq\vphi(p)+C_5e^{C_6\sup_{M}h}.
\end{align*}
So the upper bound of $\vphi$ follows from \lemref{long time: sup varphi}, i.e.
\begin{align*}
\sup_{M}\vphi\leq \underline{\vphi}+C_3.
\end{align*}
\begin{thm}(Chen-He \cite{MR2405167})
Suppose that the Ricci curvature $Ric_\vphi$ is uniformly bounded above.
Then there exist two constants $C_8$ and $C_9$ such that
\begin{align*}
n+\tri\vphi\leq
C_8e^{C_9\cdot\osc_M\vphi+\frac{\sup_M
h}{n-1}}.
\end{align*}
\end{thm}
Working with normal coordinates, we have
\begin{align*}
\frac{1}{1+\vphi_{i\bar{i}}}
=\frac{\prod_{j\neq
i}(1+\vphi_{j\bar{j}})}{\prod_{i}(1+\vphi_{i\bar{i}})}
\leq(\frac{n+\tri\vphi}{n-1})^{n-1}e^{-h}
\leq
C_{10}e^{C_{11}\osc_M\vphi+\osc_Mh}.
\end{align*}
So the metrics are all equivalent, that is
\begin{align*}
C_{10}e^{-C_{11}\osc_M\vphi-\osc_Mh}g_{i\bar{j}}\leq
g_{\vphi i\bar{j}}\leq
C_8e^{C_9\osc_M\vphi+\frac{\sup_M
h}{n-1}}g_{i\bar{j}}.
\end{align*}
By the Ricci curvature bounds
\begin{align*}
-C_1g_{\vphi i\bar{j}}\leq R_{\vphi i\bar{j}}\leq C_2 g_{\vphi
i\bar{j}},
\end{align*}
we have that $\tri h$ is bounded; namely
\begin{align*}
-C_2(n+\tri\vphi)+S\leq\tri h=-g^{i\bar{j}}R_{\vphi
i\bar{j}}+S\leq C_1(n+\tri\vphi)+S.
\end{align*}
This together with the fact that $\sup_{M\times[0,T]}|h|\leq C$ implies that $h\in
W^{2,p}(M)$ for any $p>1$ and $t\in[0,T]$. Then the Evans-Krylov estimate \cite{MR649348}\cite{MR661144}
shows that $\vphi$ has uniform $C^{2,\a}$ bound. The argument we made so far is summarized in the statement here below.
\begin{prop}\label{long time: 2+a}
Let $T$ be a finite time.
If $g_{\vphi}(t)$ is a solution of the pseudo-Calabi flow with
$Ric_{\vphi}(t)$ uniformly bounded for all time $t\in[0,T]$, then
there exists a constant $C$ depending on $T$ such that
$\sup_{[0,T]}|\vphi(t)|_{C^{2,\a}}\leq C$ for some $0<\a<1$.
\end{prop}

\section{Pseudo-Calabi flow
in the space of K\"ahler metrics}\label{PCF on SoKM}
Suppose that $\vphi(t)$ for $0\leq t<T$ is the solution of \eqref{notations: PCF} given by \thmref{short time: main} and $T$ is the maximal existence time.
In this section, we consider the following system of equations, \begin{equation}\label{PCF on SoKM: MPCF}
  \left\{
   \begin{aligned}
\frac{\p}{\p t}\psi
&=\log\frac{\om^n_\psi}{\om^n}-P(\psi)-\bar h+\bar P,\\
\tri_\psi P(\psi)&=tr_{\psi}{Ric(\om)}-\ul S,\\
\psi(0)&=\vphi_0,  \\
   \end{aligned}
  \right.
\end{equation}
with the normalization condition
\begin{align*}
\int_Me^{P}\om^n=\int_Me^{\frac{\p\psi}{\p t}
+P+\bar h-\bar P}\om^n=Vol(M).
\end{align*}
Here $\bar h=\frac{1}{V}\int_M\log\frac{\om^n_\psi}{\om^n}\om^n_\psi$ and $\bar P=\frac{1}{V}\int_MP(\psi)\om^n_\psi$.
Actually \eqref{PCF on SoKM: MPCF} is obtained by replacing $\vphi$ with
$$\psi=\vphi+\int_0^t(-\bar h+\bar P)ds$$ in \eqref{notations: pam of PCF}.
Since $\p_tI(\psi)=\int_M\p_t\psi\om_\vphi^n=0$ (see \eqref{notations: I}), if we further assume $I(\vphi_0)=0$ then $\psi$ always
stays in $\mathcal{H}_0$.

We assume that $M$ admits a cscK metric $\om$.
We choose $\om$ as the reference metric.
Then we shall show that
if $\vphi_0$ is in a sufficiently small neighborhood of the zero function, then $\vphi(t)$ can be extended and $\om_\vphi$ always stays in a small neighborhood of $\om$. Before we go into the details of the proof, we cite the theorems we will use later.

Recall the explicit form
of $K$-energy in Chen \cite{MR1772078} and Tian \cite{MR1787650}
\begin{align*}
\nu_\om(\vphi)
&=\frac{1}{V}\int_M\log\frac{\om^n_\vphi}{\om^n}\om_\vphi^{n}
+\frac{\as}{V}\sum_{i=0}^n\frac{n!}{(i+1)!(n-i)!}
\int_{M}\vphi\om^{n-i}\wedge(\p\bar\p\vphi)^i\\
&-\frac{1}{V}\sum_{i=0}^{n-1}\frac{n!}{(i+1)!(n-i-1)!}\int_{M}\vphi
Ric\wedge\om^{n-1-i}\wedge(\p\bar\p\vphi)^{i}.
\end{align*} So the $K$-energy is well defined for any $L^\infty$ K\"ahler metrics. Its derivative is
\begin{align*}
\nu'_\om(v)
=\int_M\log\frac{\om^n_\vphi}{\om^n}\tri_\vphi v\om_\vphi^n
-\frac{1}{(n-1)!}\int_Mv Ric(\om_0)\wedge\om_\vphi^{n-1}
+\ul S\int_Mv\om^n_\vphi.
\end{align*} Here $v$ is the infinitesimal variation of $\vphi$.

Chen-Tian further proved that
\begin{thm}(Chen-Tian \cite{MR2434691})\label{PCF on SoKM: CT K LB}
Let $M$ be a compact K\"ahler manifold with a cscK metric $\om$. Then
$\nu_\om(\vphi)\geq0$ for any $\vphi$ with $\om_\vphi>0$.
\end{thm}
They also proved the uniqueness of the extremal metrics.
\begin{thm}(Chen-Tian \cite{MR2434691})\label{PCF on SoKM: ex uni}
Let $(M, [\om])$ be a compact K\"ahler manifold with a K\"ahler class $[\om]\in
H^2(M,R)\cap H^{1,1}(M,C)$. Then there is at most one extremal K\"ahler metric with K\"ahler class $[\om]$
modulo holomorphic transformations. Namely, if $\om_1$ and $\om_2$ are two extremal K\"ahler metrics with
the same K\"ahler class, then there is a holomorphic transformation $\sigma$ such that
$\om_1=\sigma^\ast\om_2.$
\end{thm}
\subsection{$M$ admits no holomorphic vector field}
Suppose that $\vphi$ is the solution of \eqref{PCF on SoKM: MPCF} and $T$ is the time when
\begin{align}
|\vphi|_{C^{2,\a}}\leq\eps_1 \text{ for all }t\in[0,T].
\end{align}
Letting $t_0$ be a small time, we apply \thmref{regularity of PCF: reg of PCF} to obtain the higher order uniform bound of $\vphi$; namely
\begin{align*}
|\psi_s|_{C^{k,\a}(M)}\leq C(k,\eps_1,g,t_0) \text{ for all }t\in [T-t_0,T+t_0].
\end{align*}
Then we introduce the space
$$\mathcal{S}
=\{\vphi\vert|\vphi|_{C^{2,\a}}\leq\eps_1;
|\vphi|_{C^{k,\a}(M)}\leq C(k,\eps_1,g,t_0)\}.$$
Clearly, $0\in\mathcal{S}$.
In this subsection we are going to prove the following theorem.
\begin{thm}\label{PCF on SoKM: no holo}
Assume $M$ admits a cscK metric $\om$
and has no holomorphic vector fields.
For any $\eps_1>0$, there exits $\eps_0>0$ such that if $|\vphi_0|_{C^{2,\a}(M)}\leq\eps_0$, the lifespan of the solution is $T=\infty$ and we have that $|\psi(t)|_{2,\a}<\eps_1$ for all
$t\in[0,+\infty)$.
\end{thm}
\begin{proof}
Suppose that the conclusion fails,
then there must exist a sequence of initial data
$\vphi^0_{s}$ such that
\begin{align*}
|\vphi^0_s|_{C^{2,\a}}\leq\frac{1}{s}.
\end{align*}
By virtue of the short time
existence \thmref{short time: main},
we get a sequence of solutions $\psi_s(t)$ satisfying the
equations \eqref{PCF on SoKM: MPCF} with
$\psi_s(0)=\vphi^0_s$. Let $T_s$ be the first time
that
\begin{align}\label{PCF on SoKM: no hol ass}
|\psi_s(T_s)|_{C^{2,\a}}=\eps_1\text{ and }
|\psi_s(t)|_{C^{2,\a}}<\eps_1 \text{ on }[0,T_s).
\end{align} According to \thmref{con dpd: main}, we have that
$\inf_{s}T_{s}>0$ uniformly and there
is a uniformly small time $t_0$ such that
\begin{align*}
|\psi_s|_{C^{2,\a}(M)}\leq C(\eps_1,t_0), \forall
t\in[T_s-t_0,T_s+t_0].
\end{align*}
Moreover,
from the regularity \thmref{regularity of PCF: reg of PCF},
we obtain the
higher order uniform bound of the sequence of the solutions
\begin{align}\label{PCF on SoKM: hb}
|\psi_s|_{C^{k,\a}(M)}\leq C(k,\eps_1,g,t_0), \forall
t\in[T_s-\frac{t_0}{2},T_s+\frac{t_0}{2}].
\end{align}
Therefore we can choose a subsequence of $\phi_s=\psi_s(T_s)$;
we use the
same index for convenience, so that
\begin{align*}
\phi_s\rightarrow\phi_\infty\text{ in } C^{k,\a},\forall k\geq0.
\end{align*} Since \eqref{PCF on SoKM: no hol ass} still holds in the limit, we have
\begin{align}\label{PCF on SoKM: no hol ass lim}
|\phi_\infty|_{C^{2,\a}}=\eps_1.
\end{align}
Note that the $K$-energy is well defined along $\psi_s(t)$.
By using \thmref{PCF on SoKM: CT K LB} and the decrease of the $K$-energy along the flow, we conclude that
\begin{align*}
0\leq\nu_\om(\phi_s)\leq\nu_\om(\psi_s(0))\leq\frac{C}{s}
\end{align*}
which implies
\begin{align*}
\lim_{s\rightarrow\infty}\nu_\om(\psi_s)=\nu_\om(\psi_\infty)=0.
\end{align*}
So \thmref{PCF on SoKM: ex uni} implies that $\phi_\infty=const$.
Furthermore from $I(\phi_\infty)=0$, we deduce that $\phi_\infty=0$.
This
contradicts \eqref{PCF on SoKM: no hol ass lim}.
\end{proof}

Since $|\psi(t)|_{C^{2,\a}}\leq\eps_1$ is uniformly
bounded, we have that $|\psi(t)|_{C^k}\leq C_k$ for any
$k\geq3$ away from $t=0$. Similarly to the argument above, for any sequence $t_i$ we can extract a subsequence (still denoted by $\psi_{t_i}$) such that $\psi(t_j)$ converge to a limit function $\psi_\infty$ in $C^{\infty}$ norm. Also the limiting metric $\om_{\psi_\infty}$ is a cscK metric. Since we assume that $M$  admits no holomorphic vector field, according to \thmref{PCF on SoKM: ex uni} we deduce that $\psi_\infty=0$ from the normalization condition $I(\psi_\infty)=0$.
Therefore the pseudo-Calabi flow converges to the
original cscK metric, since ${t_i}$ is chosen randomly.
In conclusion, the picture in the space of K\"ahler metric $\mathcal{H}_0$
is that the pseudo-Calabi flow
will shrink to the unique cscK metric if the initial potential is close around it.

\subsection{$M$ admits holomorphic vector fields}
When $M$ admits holomorphic vector fields,
the
contradiction argument is much more sophisticated.
We denote the subset of $\mathcal{H}_0$ that contains the potentials of all cscK metrics by
\[
\mathcal{E}_0=\{\rho\in C^\infty(M,R)|\sigma^\ast\om
=\om+\frac{\sqrt{-1}}{2}\p\bar\p\rho\text{ and }I(\vphi)=0
\text{ for any }\sigma\in Aut_0(M)\}.
\]
Mabuchi \cite{MR909015} proved that it is a finite-dimensional totally geodesic submanifold in $\mathcal{H}_0$.
Then for each $\vphi\in\mathcal{H}_0$,
there exists a unique $\rho$ that
minimizes the distance from $\vphi$ to $\mathcal{H}_0$, i.e.
$$dist(\vphi,\rho)=dist(\vphi,\mathcal{E}_0).$$
Meanwhile there is a $\sigma\in Aut_0(M)$
such that
$\sigma^\ast\om_\vphi=\om+\frac{\sqrt{-1}}{2}\p\bar\p(\vphi(\sigma)+\rho)\in [\om]$.
We assume $\psi(0)$ stays in the complement of $\mathcal{H}_0$;
otherwise the flow keeps fixed. Furthermore we can assume that $\rho_0=0$;
if not,
we can replace the background metric $\om$ with  $\om_{\rho_0}$.
\begin{lem}\label{PCF on SoKM: dis imp norm}
There exists a small constant $\eps$ and a positive constant $C_2$ such
that if $\rho$ satisfies $d(0,\rho)\leq\eps$, then
$|\rho|_{C^{3,\a}}\leq C_2\eps$.
\end{lem}
\begin{proof}
In the Riemannain manifold $\mathcal{E}_0$, any small $\eps$ neighborhood near $\rho=0$ can be pulled back
by the exponential map $exp_0$ to the tangent space
$T_0(\mathcal{E}_0)$ near $0$.
Denote $\psi=exp_0^{-1}(\rho)$.
Note that all norms on a finite-dimensional vector space are equivalent, so the norm induced by the distance on $T_0(\mathcal{E}_0)$ is equivalent to the $C^{3,\a}$ norm (here the point in $T_0(\mathcal{E}_0)$ is also a function on $M$).
So $d(0,\rho)\leq\eps$ implies that $|exp_0^{-1}(\rho)|_{C^{3,\a}}$ is bounded by $C_1\eps$.
Furthermore, let $\eps$ be small enough such that $\psi$ is a diffeomorphism in the $\eps$ neighborhood near $\rho=0$;  then there exists a constant $C_2$ such that $|\rho|_{C^{3,\a}}\leq C_2\eps$. Hence this lemma follows.
\end{proof}
\begin{rem}
In fact, we can improve the above conclusions
in \lemref{PCF on SoKM: dis imp norm} for $C^{k}$
of any fix $k\geq0$.
\end{rem}
\begin{lem}\label{PCF on SoKM: gauge bound}
Suppose that $|\vphi|_{C^{2,\a}}\leq\eps_1$ for some small constant $\eps_1$ depending on $\eps$. Then there exists a constant $C$ depending on $\eps_1$ such that
$|\rho|_{C^{3,\a}}\leq C$
and $|\sigma|_h\leq C$.
Here $h$ is the left invariant metric in $Aut(M)$.
\end{lem}
\begin{proof}
we define a path by $\gamma_t=t\varphi-I(t\varphi)\in \mathcal{E}_0$
for $0\leq t\leq1$.
It is obvious that this path stays in $\mathcal{H}_0$.
Then since $|\varphi|_{C^{2,\alpha}}\leq\eps_1$, we have
\begin{align*}
d(0,\varphi)
&\leq L(\gamma_t)=\int_0^1(\int_M(\frac{\partial\gamma_t}{\partial t})^2
\omega^n_{\gamma_t})^{\frac{1}{2}}dt\\
&=\int_0^1(\int_M(\varphi-\partial_tI(t\varphi))^2
\omega^n_{\gamma_t})^{\frac{1}{2}}dt\leq C_3\epsilon_1.
\end{align*}
Moreover,
since $\rho$ realizes the shortest distance from $\vphi$ to $\mathcal{H}_0$, by using the triangle inequality we obtain
\begin{align}\label{dis}
d(0,\rho)
&\leq d(0,\vphi)+d(\vphi,\rho)\nonumber\\
&\leq 2d(0,\vphi) \leq C_3\eps_1.
\end{align}
Applying \lemref{PCF on SoKM: dis imp norm} with $\eps=C_3\eps_1$ we have $|\rho|_{C^{3,\a}}\leq C_2C_3\eps_1$.
Furthermore, from Lemma 4.6 in Chen-Tian \cite{MR2219236},
we obtain $|\sigma|_h$ is also bounded.
Here $h$ is the left invariant metric in $Aut(M)$. Therefore the lemma holds for some constant $C$.
\end{proof}
Now we prove the invariance of the $K$-energy.
\begin{lem}\label{PCF on SoKM: energy decay}
$\nu(\om,\om_{(\sigma^{-1})^\ast(\vphi-\rho)})=\nu(\om,\om_{\vphi})$.
\end{lem}
\begin{proof}
Since the $K$-energy is invariant under the holomorphic transformation, we get $$\nu(\om,\om_{(\sigma^{-1})^\ast(\vphi-\rho)})
=\nu(\sigma^\ast\om,\om_{\vphi})=\nu(\om_\rho,\om_{\vphi}).$$
Then the $1$-cocycle condition of the $K$-energy (Theorem (2.4) in \cite{MR867064}) gives
$$\nu(\sigma^\ast\om,\om_{\vphi})=\nu(\om_{\rho},\om)+\nu(\om,\om_{\vphi}).$$
Since both $\om$ and $\om_{\rho}$ are cscK metrics, the lemma follows from $\nu(\om_{\rho},\om)=0$.
\end{proof}
\begin{lem}\label{PCF on SoKM: norm bound by enery}
For any $\eps>0$, There exists a small constant $o$
such that for any $\vphi\in\mathcal{S}$,
if $\nu_{\om}(\vphi)\leq o$,
then $|(\sigma^{-1})^\ast(\vphi-\rho)|_{C^{2,\a}}<\eps$.
\end{lem}
\begin{proof}
 If the conclusion fails,
 we assume there is a sequence $\vphi_s$ with
 \begin{align*}
|\vphi_s|_{C^{2,\a}(M,g)}\leq\eps_1,
|\vphi|_{C^{k,\a}(M)}\leq C(k,\eps_1,g,t_0),
\text{ and }\nu_{\om}(\vphi_s)\leq\frac{1}{s}
\end{align*} such that
\begin{equation}\label{PCF on SoKM: contra assum}
|(\sigma_s^{-1})^\ast(\vphi_s-\rho_s)|_{C^{2,\a}}\geq\eps_1.
\end{equation}
We denote $\hat\vphi_s=(\sigma_s^{-1})^\ast(\vphi_s-\rho_s)$.
After making use of \lemref{PCF on SoKM: gauge bound}
and \lemref{PCF on SoKM: energy decay},
we can choose some subsequence of $\hat\vphi_s$ such that
$$\hat\vphi_s\rightarrow\hat\vphi_\infty\in C^{l} \text{ for any }l\geq 0\text{
and }\nu_{\om}(\hat\vphi_\infty)=0.$$
Therefore we have $\hat\vphi_\infty\in\mathcal{E}_0$ by Chen-Tian \cite{MR2434691}.
And then $\vphi_\infty\in\mathcal{E}_0$.

We claim that $d(\vphi_\infty,\rho_\infty)=0$.
Otherwise there is some sufficient large $N$ such that,
for any $s>N$, the sequence
$d(\vphi_s,\rho_s)=d(\vphi_s,\mathcal{E}_0)$
has positive lower bound.
From the fact that the distance function is at least $C^1$
we have $d(\vphi_\infty,\mathcal{E}_0)>0$,
that contradicts $\vphi_\infty\in \mathcal{E}_0$.

This claim
implies $\hat\vphi_\infty=0$. But \eqref{PCF on SoKM: contra assum} gives
$|\hat\vphi_\infty|_{C^{2,\a}}\geq\eps_1>0$. It is a contradiction.
\end{proof}
\begin{rem}
In fact, we can improve the above conclusion to get
$$|(\sigma^{-1})^\ast(\vphi-\rho)|_{C^{k,\a}}<\eps$$
for any $k\geq0$.
Then combining with \thmref{ed of PCF: ed},
we can show that the solution will stay in a small neighborhood
and exponentially decay to a unique cscK metric
by directly obtaining the estimate of the solution.
\end{rem}
\begin{rem}
We use the theorem of Chen-Tian \cite{MR2434691} that the cscK metric is the global minimizer of the $K$-energy here. Actually, we only need a local version. I.e. the cscK metric is the local minimizer of the $K$-energy. That can be proved by using the non-negativeness of the hessian of the $K$-energy and the geometry of the critical submanifold.
\end{rem}
\begin{thm}\label{PCF on SoKM: hol}
Assume that $M$ admits a cscK metric $\om$
and has nontrivial holomorphic vector fields.
If $|\psi_0|_{C^{2,\a}(M)}\leq\eps_0$,
then there exists a holomorphic transformation $\varrho(t)$ such that $\varrho(t)^\ast \om_{\psi_{t}}$ stays
in a small neighborhood of $\om$.
Moreover,
for any sequence $g_{\psi_{t_i}}$, one can extract a subsequence $g_{\psi_{{t_i}_j}}$ such that $\varrho_{j}^\ast g_{\psi_{{t_i}_j}}$
converges to a cscK metric.
\end{thm}
\begin{proof}
According to the short time existence \thmref{short time: main},
we assume that $T$ is the first time when the following holds
$$|\psi|_{C^{2,\a}}<\eps_1\text{ on }[0,T)
\text{ and } |\psi(T)|_{C^{2,\a}}
=\eps_1.$$ Now there are two cases.
If $\psi(T)$ is  a cscK metric,
then the flow will stop right at $T$ and our theorem is proved.
Otherwise, we will extend the flow as follows.
By virtue of \thmref{regularity of PCF: reg of PCF}, we obtain
the higher order uniform bound of the solutions
\begin{align*}
|\psi|_{C^{k,\a}(M)}\leq C(k,\eps_1,g,t_0), \forall
t\in[T-\frac{t_0}{2},T+\frac{t_0}{2}].
\end{align*}
So we have $\psi(T)\in\mathcal{S}$.

Now we firstly choose $\eps_0$ small enough to guarantee
$$\nu_\om(\psi_0)\leq o.$$
Since the $K$-energy is decreasing along the pseudo-Calabi flow, we get
so $$\nu_\om(\psi(T))\leq\nu_\om(\psi_0).$$ Let $\sigma$ be the projection of $\psi(T)$ onto $\mathcal{H}_0$.
Then due to \lemref{PCF on SoKM: norm bound by enery} we obtain
\begin{equation}\label{PCF on SoKM: itr small norm}
|(\sigma^{-1})^\ast(\psi(T)-\rho_1)|_{C^{2,\a}}<\eps_0.
 \end{equation}
Next we show the equation is invariant
under the holomorphic transformation.
By \eqref{PCF on SoKM: MPCF},
$\psi(t)-\rho_1$ is the solution of
\begin{equation}\label{PCF on SoKM: transform equ}
   \begin{cases}
\frac{\p}{\p  t}\psi
=\log\frac{\om^n_{\psi+\rho_1}}{\om^n}
-P(\psi+\rho_1)
-\overline{\log\frac{\om^n_{\psi+\rho_1}}{\om^n}}
+\overline{P(\psi+\rho_1)},\\
\tri_{\psi+\rho_1}P(\psi+\rho_1)
=tr_{\psi+\rho_1}{Ric(\om)}-\ul S.  \\
   \end{cases}
\end{equation}
Note that $\om_{\rho_1}$ is a cscK metric satisfying
 \begin{equation}\label{PCF on SoKM: gauge equ}
   \begin{cases}
\log\frac{\om_{\rho_1}^n}{\om^n}=P(\rho_1),\\
\tri_{\rho_1}P(\rho_1)=tr_{\rho_1}{Ric(\om)}-\ul S.
   \end{cases}
\end{equation}
Letting $P_1(\psi+\rho_1)=P(\psi+\rho_1)-P(\rho_1)$,
we have
\begin{align*}
\tri_{\psi+\rho_1}P_1(\psi+\rho_1)
&=\tri_{\psi+\rho_1}(P(\psi+\rho_1)-P(\rho_1))\\
&=tr_{\psi+\rho_1}{Ric(\om)}-\as-\tri_{\psi+\rho_1}P(\rho_1)\\
&=tr_{\psi+\rho_1}{Ric(\om_{\rho_1})}-\as.
\end{align*}
From this we obtain new equations from $t=T$
by combining \eqref{PCF on SoKM: transform equ} and
\eqref{PCF on SoKM: gauge equ} so that
$\psi(t)-\rho_1$ satisfies
 \begin{equation}\label{PCF on SoKM: extend equ 1}
   \begin{cases}
\frac{\p}{\p  t}\psi
=\log\frac{\om^n_{\psi+\rho_1}}{\om_{\rho_1}^n}
-P_1(\psi+\rho_1)
-\overline{\log\frac{\om^n_{\psi+\rho_1}}{\om_{\rho_1}^n}}
+\overline{P_1(\psi+\rho_1)},\\
\tri_{\psi+\rho_1}P_1(\psi+\rho_1)
=tr_{\psi+\rho_1}{Ric(\om_{\rho_1})}-\ul S.
   \end{cases}
\end{equation}
After taking transformation $(\sigma^{-1})^\ast$ of
\eqref{PCF on SoKM: extend equ 1} we have
$\psi_1=(\sigma^{-1})^\ast(\psi(t)-\rho_1)$ is the solution of
\begin{equation}\label{PCF on SoKM: extend equ}
  \left\{
   \begin{aligned}
\frac{\p}{\p  t}\psi_1
&=\log\frac{\om^n_{\psi_1}}{\om^n}-P_1(\psi_1)
-\overline{\log\frac{\om^n_{\psi_1}}{\om^n}}+\overline{P_1(\psi_1)},\\
\tri_{\psi_1}P_1(\psi_1)&=tr_{\psi_1}{Ric(\om)}-\as,\\
\psi_1(0)&=\psi(T)(\sigma)-\rho_1,  \\
   \end{aligned}
  \right.
\end{equation}
 with \eqref{PCF on SoKM: itr small norm} and
 $$\nu_{\om}(\psi_1(0))\leq o.$$
All averages in these equations are taken over $M$ with respect to the metric $g_{\psi_1}$
in these equations. \thmref{short time: main}
implies that \eqref{PCF on SoKM: extend equ} can be extended beyond $T$.
Finally since $\psi_1(T_1)\in\mathcal{S}$,
we can repeat the same steps as before
by induction till $\psi_s$ becomes a cscK metric at time $T$ if $T<\infty$.
If not, we deduce that the pseudo-Calabi flow has long time existence
and there is a sequence $$\psi_s(0)=\psi_s^0
=(\sigma_{s-1}^{-1})^\ast(\psi_{s-1}(T_{s-1})-\rho_{s-1})$$
such that
 \begin{align*}
 &|\psi_s^0|_{C^{k}}\leq C(k,\eps_1)
\forall l\geq0,\\
&\lim_{s\rightarrow\infty}\nu_{\om}(\psi^0_{s})
=\lim_{s\rightarrow\infty}\nu(\om,\om_{\psi_{s-1}(T_{s-1})})=0.
\end{align*}

We further define
$$\omega_{\psi(t)}
=\prod_{i=0}^{s-1}\sigma_{i}^\ast\om_{\psi_{s}(t)}\text{ on }[\sum_{i=0}^{s-1}T_{i},\sum_{i=0}^{s}T_{i}).$$
Then we obtain a solution $\psi(t)$ for all $t\geq0$.
In general, for any sequence $\{\psi_{t_j}\}$,
there is $s$ such that
$\sum_{i=0}^{s-1}T_{i}\leq t_j \leq\sum_{i=0}^{s}T_i$.
Furthermore, writing $\varrho_{j}=(\prod_{i=0}^{s-1}\sigma_{i})^{-1}$,
we have
\begin{align*}
|\varrho_{j}^\ast\om_{\psi_{t_j}}
-\om|_{C^{\a}}\leq \eps_1 \text{ and }
|\varrho_{j}^\ast\om_{\psi_{t_j}}
-\om|_{C^{k}}\leq C(k,\eps_1).
\end{align*}
Therefore all metrics are equivalent and their derivatives are bounded.
It follows that there is a subsequence (with the same notation) of
$\varrho_{j}^\ast g_{\psi_{t_j}}$ that
converges to a limit K\"ahler metric $g_\infty$
which may depend on the choice of the sequence.
Since the $K$-energy is bounded below,
we obtain
\begin{align*}
\lim_{s\rightarrow\infty}\nu(\om,\om_{\psi_{t_j}})=0.
\end{align*}
It follows that $g_\infty$ is a cscK metric.
\end{proof}
\begin{rem}
Following the same argument in Chen-Tian \cite{MR2219236},
we can extend the holomorphic transformation $\varrho$ to each $t$
so that it is Lipschitz continuous in $t$.
\end{rem}

\section{Exponential decay of the pseudo-Calabi flow}\label{ed of PCF}
Recall that $\psi(t)$ is the solution
of \eqref{PCF on SoKM: MPCF}
and $g_\phi=\varrho^\ast g_\psi$ is
the modified solution defined in \thmref{PCF on SoKM: hol}. We have already proved that $g_\phi$ always stays
in small neighborhood of $\om$
and converges to a cscK metric sequently.
\begin{defn}
We call $f$-tensor the $(1,1)$ form locally given by
$\mathrm{f}
=[P_{i\bar{j}}+R_{i\bar{j}}-R_{i\bar{j}}(\om)]
dz^i\wedge dz^{\bar j}
=f_{i\bar{j}}dz^i\wedge dz^{\bar j}$.
\end{defn}
Since $\lim_{t_j\rightarrow\infty}S_{t_j}-\ul S=0$ for arbitrary subsequence $\{t_j\}$,
we have $$\lim_{t\rightarrow\infty}S_{t}-\ul S=0.$$
Moreover, the uniform bound of $|Rm(g_\phi)|_{g_\phi}$ gives rise to the uniform bound of $|Rm(g_\psi)|_{g_\psi}$.
These facts are important to prove the exponential decay of the energy functionals.
\begin{lem}\label{ed of PCF: ed}
The following formula holds
\begin{align}\label{ed of PCF: mu1}
&\frac{d}{dt}\frac{1}{V}\int_M|\nabla\dot\psi|^2\om_\psi^n
=-2\frac{1}{V}\int_M(L_t\dot\psi,\dot\psi)\om_\psi^n\\
&+\frac{1}{V}\int_M(-P^{i\bar{j}}-R^{i\bar{j}}
+R^{i\bar{j}}(\om))\dot\psi_i\dot\psi_{\bar j}\om_\psi^n
-\frac{1}{V}\int_M|\nabla\dot\psi|^2(S-\ul{S})\om_\psi^n.\nonumber
\end{align}
\end{lem}
\begin{proof}
Differentiating \eqref{PCF on SoKM: MPCF} with respect to $t$ we have
\begin{align}\label{ed of PCF: p_t equ}
\ddot\psi=\triangle_\psi\dot\psi
-\dot P-\ul{\triangle_\psi\dot\psi}-\ul{\dot P}.
\end{align} Then we get
\begin{align*}
\frac{d}{dt}|\nabla\dot\psi|^2
&=f^{i\bar{j}}\dot\psi_i\dot\psi_{\bar j}+g_\psi^{i\bar
j}\ddot\psi_i\dot\psi_{\bar j}
+g_\psi^{i\bar
j}\dot\psi_i\ddot\psi_{\bar j}\\
&=f^{i\bar j}\dot\psi_i\dot\psi_{\bar j}
+g_{\psi}^{i\bar j}\triangle_\psi\dot\psi_i\dot\psi_{\bar j}
-g_\psi^{i\bar
j}\dot P_i\dot\psi_{\bar j}
+g_{\psi}^{i\bar j}\dot\psi_i\triangle_\psi\dot\psi_{\bar j}
-g_\psi^{i\bar
j}\dot\psi_i\dot P_{\bar j}.
\end{align*}
On the other hand
since
$$(\tri_\psi\psi)_i=\psi_{k\bar{k}i}
=\psi_{\bar{k}ki}=\psi_{\bar{k}ik}
=\psi_{i\bar{k}k}=\triangle_\psi\psi_i,$$
we obtain the Laplacian of $|\nabla\dot\psi|^2$ given by
$$\triangle_\psi|\nabla\dot\psi|^2
=R^{i\bar{j}}\dot\psi_i\dot\psi_{\bar{j}}
+g_\psi^{i\bar{j}}\triangle_\psi\dot\psi_i\dot\psi_{\bar{j}}
+g_\psi^{i\bar{j}}\dot\psi_i\triangle_\psi\dot\psi_{\bar{j}}
+|\dot\psi_{ij}|^2+|\dot\psi_{i\bar{j}}|^2.$$
In the following we use the upper index
to represent the raise of the index by means of the metric $g_\psi$.
Then combining these two identities we have the following evolution equation,
\begin{align}\label{ed of PCF: p t equ p}
&\frac{d}{dt}\frac{1}{V}\int_M|\nabla\dot\psi|^2\om_\psi^n\nonumber\\
&=\frac{1}{V}\int_M[(f^{i\bar{j}}-R^{i\bar{j}})\dot\psi_i\dot\psi_{\bar j}
-|\dot\psi_{ij}|^2-|\dot\psi_{i\bar{j}}|^2
-g_\psi^{i\bar j}\dot P_i\dot\psi_{\bar{j}}
-g_\psi^{i\bar j}\dot\psi_i\dot{P}_{\bar{j}}]\om_\psi^n\nonumber\\
&+\frac{1}{V}\int_M|\nabla\dot\psi|^2
\tri_\psi\dot\psi\om_\psi^n.
\end{align}
Differentiating the second equation in \eqref{PCF on SoKM: MPCF} we get
$$g_\psi^{i\bar{j}}\dot P_{i\bar{j}}
=\dot g_\psi^{i\bar{j}}(-P_{i\bar{j}}+R_{i\bar{j}}(\om))
=f^{i\bar{j}}(-P_{i\bar{j}}+R_{i\bar{j}}(\om)).$$ So using \eqref{notations: pam of PCF} and the fact that
$-P_{i\bar{j}}+R_{i\bar{j}}(\om)$ is a harmonic form, we have
\begin{align}\label{ed of PCF: dot p}
-\int_Mg_\psi^{i\bar j}\dot P_i\dot\psi_{\bar{j}}\om_\psi^n
&=\int_M\tri_\psi\dot P\dot\psi\om_\psi^n
=\int_M(f^{i\bar{j}}(-P_{i\bar{j}}+R_{i\bar{j}}(\om)))
\dot\psi\om_\psi^n\nonumber\\
&=\int_M((-P^{i\bar{j}}+R^{i\bar{j}}(\om)))
\dot\psi_i\dot\psi_{\bar j}\om_\psi^n.
\end{align}
Combining \eqref{ed of PCF: p t equ p}, \eqref{ed of PCF: dot p} with the following equations
$$f_{i\bar j}=P_{i\bar j}-h_{i\bar j}\text{ and }
\tri_\psi\frac{\p}{\p t}\psi
=-\tri_\psi f=-(S-\ul{S}),$$
we obtain
\begin{align*}
&\frac{d}{dt}\frac{1}{V}\int_M|\nabla\dot\psi|^2\om_\psi^n
=-\frac{1}{V}\int_M|\dot\psi_{ij}|^2\om_\psi^n
-\frac{1}{V}\int_M|\dot\psi_{i\bar{j}}|^2\om_\psi^n\nonumber\\
&+\frac{1}{V}\int_M(-P^{i\bar{j}}+R^{i\bar{j}}(\om))
\dot\psi_i\dot\psi_{\bar j}\om_\psi^n
-\frac{1}{V}\int_M|\nabla\dot\psi|^2(S-\ul{S})\om_\psi^n.
\end{align*}
Last we insert the Ricci identity, i.e.
$$\int_M|\dot\psi_{i\bar j}|^2\om_\psi^n
=\int_M|\dot\psi_{ij}|^2\om_\psi^n
+\int_MR^{i\bar{j}}\dot\psi_i\dot\psi_{\bar j}\om_\psi^n$$
into the above differential inequality and obtain immediately,
\begin{align*}
&\frac{d}{dt}\frac{1}{V}\int_M|\nabla\dot\psi|^2\om_\psi^n\\
&=-2\frac{1}{V}\int_M|\nabla\nabla\dot\psi|^2\om_\psi^n
+\frac{1}{V}\int_M(-P^{i\bar{j}}-R^{i\bar{j}}
+R^{i\bar{j}}(\om))\dot\psi_i\dot\psi_{\bar j}\om_\psi^n\\
&-\frac{1}{V}\int_M|\nabla\dot\psi|^2(S-\ul{S})\om_\psi^n\\
&=-2\frac{1}{V}\int_M(L_t\dot\psi,\dot\psi)\om_\psi^n
-\frac{1}{V}\int_Mf^{i\bar{j}}\dot\psi_i\dot\psi_{\bar j}\om_\psi^n
-\frac{1}{V}\int_M|\nabla\dot\psi|^2(S-\ul{S})\om_\psi^n.
\end{align*}
\end{proof}

Since $g_\psi$ converges to a cscK metric,
we have that both $\mathrm{f}$
and $S-\ul S$ tend to zero. As a result we deduce that for any small $\eps$,
\begin{align}\label{PCF on SoKM:mu1 small}
\frac{d}{dt}\frac{1}{V}\int_M|\nabla\dot\psi|^2\om_\psi^n
\leq-2\frac{1}{V}\int_M(L_t\dot\psi,\dot\psi)\om_\psi^n
+\eps\frac{1}{V}\int_M|\nabla\dot\psi|^2\om_\psi^n.
\end{align}
Moreover,
the Futaki invariant implies that
\begin{align}\label{PCF on SoKM:Futaki}
0&=F(X)=\int_MX(f)\om_\psi^n
=\int_M(\theta_X(\psi))_if_{\bar i}\om_\psi^n\nonumber\\
&=-\int_M\theta_X(\psi)(S-\ul S)\om_\psi^n
=-\int_M(\theta_X(\psi))_i\dot\psi_{\bar i}\om_\psi^n
\end{align}
for any $X=\uparrow\bar\p\theta_X\in\eta(M)$.

Let $L_t$ be the Lichnerowicz operator. It is a positive semidefinite,
self-adjoint operator and $L_t\psi=0$ if and only if
$\nabla\psi=\uparrow\bar\p\psi$ is a holomorphic vector field.
We are going to obtain the first eigenvalue of $L_t$ with
regard to the metric $g_{\phi(t)}$ first.
We define the set
\begin{align*}
A_t&=\{f\in C^\infty_R(M)\vert
\int_Mf\om_\phi^n=0;\\
&\int_M(\theta_X(\phi))_i f_{\bar i}\om_\phi^n=0,
\forall X=\uparrow\bar\p\theta_X(\phi)\in\eta(M)\}.
\end{align*}
Then \eqref{PCF on SoKM:Futaki} implies $f_\phi\in A_t$ for $\tri_\phi f_\phi=S_\phi-\ul S$.
We further define
\begin{align*}
\l(\phi)&=\inf_{f\in A_t}\{c|\int_M|\nabla\nabla f|^2\om_\phi^n \geq
c\int_M|\nabla f|^2\om_\phi^n\}\\
&=\inf_{f\in A_t}\{c|c\leq\int_M|\nabla\nabla f|^2\om_\phi^n;
\underline{}\int_M|\nabla f|^2\om_\phi^n=1\}.
\end{align*}

We prove a similar lemma to Chen-Li-Wang \cite{MR2481736}.
\begin{lem}\label{ed of PCF:eigenvalue}
We have the uniform lower bound of
the Lichnerowicz operator $L_t$ i.e.
$$\l>0\text{, }\forall t\geq0.$$
\end{lem}
\begin{proof}
We prove the lemma by the contradiction method. If the conclusion fails,
then we can choose a sequence $t_s$
such that
\begin{align*}
&\int_M|\nabla_s\nabla_s f_s|^2\om_{\phi_s}^n
<\l(\phi_s)\int_M|\nabla_s f_s|^2\om_{\phi_s}^n
=\frac{1}{s}\int_M|\nabla_s f_s|^2\om_{\phi_s}^n,
\int_Mf_s\om_{\phi_s}^n=0\\
&\text{and }\int_M(\theta_X(\phi_s))_i (f_s)_{\bar i}\om_{\phi_s}^n=0,
\forall X=\uparrow\bar\p\theta_X(\phi_s)\in\eta(M).
\end{align*}
We further scale $f_s$ such that
\begin{align}\label{ed of PCF: bound of fs}
&\int_M|\nabla_s\nabla_s f_s|^2\om_{\phi_s}^n <\frac{1}{s},
\int_M|\nabla f_s|^2\om_{\phi_s}^n=1,
\int_Mf_s\om_{\phi_s}^n=0\\
&\text{and }\int_M(\theta_X(\phi_s))_i (f_s)_{\bar i}\om_{\phi_s}^n=0,
\forall X=\uparrow\bar\p\theta_X(\phi_s)\in\eta(M).\nonumber
\end{align}
Recall that \thmref{PCF on SoKM: hol} gives a subsequence of $g_{\phi_s}$ satisfying
$$\lim_{j\rightarrow\infty}g(\phi(t_{s_j}))
=g(\phi_\infty)\text{ in }C^{l} \text{ for any }l\geq0.$$
Moreover, combining Ricci identity
$$\int_M|(f_s)_{i\bar j}|^2\om_{\phi_s}^n
=\int_M|(f_s)_{ij}|^2\om_{\phi_s}^n
+\int_MR^{i\bar{j}}(f_s)_i(f_s)_{\bar j}\om_{\phi_s}^n$$
and by the uniform boundness of Ricci curvature we obtain
$$\int_M|\nabla\bar\nabla f_s|^2\om_{\phi_s}^n \leq\frac{1}{s}+C(Ric).$$
This inequality together with \eqref{ed of PCF: bound of fs} and the Poincar\'{e} inequality
provides a $W^{2,2}$-weak convergent subsequence of $f_s$ such that
\begin{align*}
f_s\rightharpoonup f_\infty\nonumber \text{ in } W^{2,2}.
\end{align*}
In addition, the Sobolev imbedding theorem implies
\begin{align*}
\nabla_s f_s\rightarrow\nabla_\infty f_\infty \text{ in } L^2\text{ and }
f_s\rightarrow f_\infty \text{ in } L^2. \nonumber
\end{align*}
Accordingly, by the assumption we have
\begin{align}\label{ed of PCF: eigenvalue ass}
&\int_M|\nabla_\infty\nabla_\infty f_\infty|^2\om_\infty^n
\leq\liminf_{s\rightarrow\infty}
\int_M|\nabla_\infty\nabla_\infty f_s|^2\om_{\phi_\infty}^n\\
&\leq C\liminf_{s\rightarrow\infty}
\int_M|\nabla_s\nabla_s f_s|^2\om_{\phi_s}^n=0,\nonumber\\
&\lim_{s\rightarrow\infty}\int_M|\nabla_s f_s|^2\om_{\phi_s}^n
=\int_M|\nabla\infty f_\infty|^2\om_{\phi_\infty}^n=1,\nonumber\\
&\lim_{s\rightarrow\infty}\int_Mf_s\om_{\phi_s}^n
=\int_Mf_\infty\om_{\phi_\infty}^n=0.\nonumber
\end{align}
On the other hand we turn to the
Futaki invariant,
\begin{align*}
0=\int_MX(f_s)\om_{\phi_s}^n
=\int_M(\theta_X(\phi_s))_i (f_s)_{\bar i}\om_{\phi_s}^n=0
\end{align*} for any $X\in \eta(M,g_\phi)$
which implies
\begin{align}\label{ed of PCF: futaki}
0=\int_MX(f_\infty)\om_\infty^n
=\int_M{\theta_X(\phi_\infty)}_i{f_\infty}_{\bar i}\om_\infty^n.
\end{align} after taking limit.
Notice that the complex structure is fixed, then we have
\begin{align*}
\eta(M,g_\infty)=\eta(M,g(\phi_s)).
\end{align*}
Together with \eqref{ed of PCF: futaki} it implies $f_\infty$
does not belong to
$$Ker L_\infty=\{\theta_X(\phi_\infty)\vert X
=\uparrow\bar\p \theta_X(\phi_\infty)\in \eta(M,g_\infty);
\int_M\theta_X(\phi_\infty)\om_{\phi_\infty}^n=0\}.$$
Consequently, we have
\begin{align}
\int_M|\nabla_\infty\nabla_\infty f_\infty|^2\om_\infty^n
>\l\int_M|\nabla_\infty f_\infty|^2\om_\infty^n=\l>0.
\end{align}
That contradicts to \eqref{ed of PCF: eigenvalue ass} and the lemma follows.
\end{proof}
Since the eigenvalue $\l$ is invariant under the holomorphic transformation,
we obtain the uniform positive lower bound
of the first eigenvalue of $g(\psi(t))$. Therefore
\lemref{ed of PCF:eigenvalue} immediately implies
\begin{align*}
\int_M\dot\psi L_t\dot\psi\om_\psi^n
\geq \l\int_M|\dot\psi|^2\om_\psi^n.
\end{align*}
Substituting this inequality into \eqref{PCF on SoKM:mu1 small}, by the Gronwall's inequality we obtain the exponential decay of the energy $\mu_1$,
\begin{align}\label{mu1}
\mu_1(t)=\frac{1}{V}\int_M|\nabla\dot\psi|^2\om_\psi^n\leq
\mu_1(0)e^{-\theta t}.
\end{align}
Moreover the inequality \eqref{mu1} together with the Poincar\'{e} inequality and the normalization condition $$\int_M\p_t\vphi\om_\vphi^n=0,$$ implies
\begin{align}\label{mu0}
\mu_0(t)=\frac{1}{V}\int_M|\dot\psi|^2\om_\psi^n
\leq\mu_0(0)e^{-\theta t}.
\end{align}
Furthermore, we can control the evolution of $\mu_l(t)=\frac{1}{V}\int_M|\nabla^l\dot\psi|^2\om_\psi^n$ by the following lemma.
\begin{lem}\label{ed of PCF: nu l ev}For any $l\geq2$ the following inequality holds
\begin{align*}
\p_t\mu_l(t)\leq C\mu_0(t).
\end{align*}
\end{lem}
\begin{proof}
We prove this lemma in real coordinates.
Let $I=(i_1,\cdots,i_l)$, $J=(j_1,\cdots,j_l)$ and $g^{IJ}=g^{i_1j_1}\cdots g^{i_lj_l}$.
Differentiating $|\nabla^l\dot\psi|^2$ with respect to $t$ and using \eqref{ed of PCF: p_t equ}, we get
\begin{align*}
\frac{d}{dt}|\nabla^l\dot\psi|^2
&=\sum g^{i_1j_1}\cdots f^{pq}\cdots g^{i_lj_l}\dot\psi_{i_1,\cdots,p,\cdots,i_l}\dot\psi_{j_1,\cdots,q,\cdots,j_l}\\
&+g_\psi^{IJ}\ddot\psi_I\dot\psi_J
+g_\psi^{IJ}\dot\psi_I\ddot\psi_J\\
&=\sum g^{i_1j_1}\cdots f^{pq}\cdots g^{i_lj_l}\dot\psi_{i_1,\cdots,p,\cdots,i_l}\dot\psi_{j_1,\cdots,q,\cdots,j_l}\\
&+g_{\psi}^{IJ}(\tri_\psi\dot\psi)_I\dot\psi_J
-g_\psi^{IJ}\dot P_I\dot\psi_J
+g_{\psi}^{IJ}\dot\psi_I(\tri_\psi\dot\psi)_J
-g_\psi^{IJ}\dot\psi_I\dot P_J.
\end{align*}
Meanwhile, a standard computation gives
$$\tri_\psi|\nabla^l\dot\psi|^2
=Rm\ast\nabla^l\dot\psi\ast\nabla^l\dot\psi
+g_{\psi}^{IJ}(\tri_\psi\dot\psi)_I\dot\psi_J
+g_{\psi}^{IJ}\dot\psi_I(\tri_\psi\dot\psi)_J
+|\nabla^{l+1}\dot\psi_{ij}|^2.$$
Combining these two equalities, we obtain
\begin{align*}
\p_t\frac{1}{V}\int_M|\nabla^l\dot\psi|^2\om_\psi^n
&=\frac{1}{V}\int_M[\sum g^{i_1j_1}\cdots f^{pq}\cdots g^{i_lj_l}\dot\psi_{i_1,\cdots,p,\cdots,i_l}\dot\psi_{j_1,\cdots,q,\cdots,j_l}\\
&-g_\psi^{IJ}\dot P_I\dot\psi_J
-g_\psi^{IJ}\dot\psi_I\dot P_J-Rm\ast\nabla^l\dot\psi\ast\nabla^l\dot\psi
-|\nabla^{l+1}\dot\psi_{ij}|^2]\om_\psi^n\\
&+\frac{1}{V}\int_M|\nabla^l\dot\psi|^2\tri_\psi\dot\psi\om_\psi^n\\
&\leq C\mu_l(t)-\mu_{l+1}(t)\\
&\leq C(\eps)\mu_1(t)+(C\eps-1)\mu_{l+1}(t).
\end{align*}
The last inequality holds by the interpolation inequality.
Then the lemma follows by \eqref{mu1} when we choose $\eps$ to be small enough.
\end{proof}
Together with \eqref{mu0} \lemref{ed of PCF: nu l ev} implies
\begin{align*}
\frac{1}{V}\int_M|\nabla^l\dot\psi|^2\om_\psi^n
\leq\mu_l(0)e^{-\theta t}
\end{align*}
for $l\geq0$ by the Gronwall's inequality again.
Since all $\varrho_{t}^\ast\om_{\psi_{t}}$ are equivalent and
the Sobolev constant and the Poincar\'{e} constant are invariant
under the holomorphic transformations, we deduce that
$\om_{\psi_{t}}$ has uniform Sobolev constant and Poincar\'{e} constant.
The Sobolev imbedding theorem implies
\begin{align}\label{PCF on SoKM:l potential}
|\dot\psi|_{C^l(g_\psi)}\leq Ce^{-\theta t}.
\end{align}
In view of the identity
\begin{align*}
\psi(t)=\psi(0)+\int_0^1\dot\psi dt,
\end{align*}
we have
\begin{align*}
|\psi(t)|_{C^0}\leq|\psi(0)|_{C^0}+Ce^{-\theta t}.
\end{align*}
Now note that the potential
and the Ricci curvature are uniformly bounded
\begin{align*}
C_1\om_\psi\leq Ric_\psi\leq C_2\om_\psi.
\end{align*}
According to Chen-He's compactness theorem \cite{MR2405167},
we obtain $|\psi|_{2,\a}$ is uniformly bounded.
Moreover the higher order bound follows from \thmref{regularity of PCF: reg of PCF}.
Consequently, \eqref{PCF on SoKM:l potential} gives rise to
$$|\psi-\psi_\infty|_{C^{l}(g_\infty)}\leq C_le^{-\theta t}.$$
Hence combining with \thmref{PCF on SoKM: no holo} and \thmref{PCF on SoKM: hol},
we obtain the following stability theorem.
\begin{thm}
For each K\"ahler
class which admits a cscK metric $\om$,
there exits a
small constant $\eps_0$ such that,
if $|\vphi_0|_{C^{2,\a}(M)}\leq\eps_0$,
then the solution to the pseudo-Calabi flow
exists for all times and
converges exponentially fast to
a unique cscK metric nearby in the K\"ahler class $[\om]$.
\end{thm}
An immediate corollary is
\begin{cor}
If the pseudo-Calabi flow converges to a cscK metric by passing to a subsequence if necessary, then the limit cscK metric (depends on the subsequence) must be unique.
\end{cor}
\begin{rem}
We hope our method of proving the stability problem may be useful in other geometric flow problems.
Since the pseudo-Calabi flow is just the K\"ahler-Ricci flow in the canonical class. By using the similar technique of the pseudo-Calabi flow,
we shall prove the stability of the K\"ahler-Ricci flow near a K\"ahler-Ricci soliton in the subsequent paper \cite{zheng-2009krf}. Meanwhile, we prove the stability of the Calabi flow near an extremal metric in \cite{zheng-2009cf}.
\end{rem}
\section{Further remarks and problems}
A Chen's conjecture \cite{Chen4-2008} says that
\begin{conj}(Chen \cite{Chen4-2008})
A global $C^{1,1}$ $K$-energy minimizer in any K\"ahler class must be smooth.
\end{conj} This conjecture has been proved in the canonical K\"ahler class via the weak K\"ahler-Ricci flow \cite{MR2348984}\cite{MR2434691}\cite{ctz-2008}\cite{songtian2009}. In general K\"ahler class, the conjecture will be verified if one proves the following question:
\begin{ques}(Chen \cite{Chencom})
Can the pseudo-Calabi flow start from a $C^{1,1}$ K\"ahler potential?
\end{ques} In Subsection \ref{estimate of linearized}, we have obtained a partial estimate related to this conjecture.
In the subsequent paper, we will relax the regularity of the initial K\"ahler potential.

As we have proved, the pseudo-Calabi flow shares some properties with the Calabi flow; recall for instance \thmref{intro: ricci}. However, even in the canonical class $C_1(M)$, they are extremely different. Within $C_1(M)$ the pseudo-Calabi flow is precisely the K\"ahler-Ricci flow. The K\"ahler-Ricci flow has many properties such as the long time existence \cite{MR799272} and the Perelman's estimate \cite{Perelman}. Chen-Tian proved the convergence of the K\"ahler-Ricci flow on K\"ahler-Einstein manifolds \cite{MR1893004}\cite{MR2219236}. Later Zhu \cite{MR2291916} extended their theorems to the K\"ahler manifolds which admit a K\"ahler-Ricci soliton. One may ask if the pseudo-Calabi flow inherits these fine properties of the K\"ahler-Ricci flow.
\begin{conj}(Chen \cite{Chencom})
The pseudo-Calabi flow has long time existence.
\end{conj}
There is a conjecture in \cite{Chen4-2008} saying
\begin{conj}(Chen \cite{Chen4-2008})
The existence of cscK metrics implies that the
$K$-energy is proper in the sense that it bounds the geodesic distance.
\end{conj} Inspired from this conjecture,
one may generalize \thmref{intro: stability} to the following question.
\begin{ques}(Chen \cite{Chencom})
On K\"ahler manifolds admit a cscK metric, does the pseudo-Calabi flow converge to a cscK metric?
\end{ques}
Note that \lemref{PCF on SoKM: dis imp norm} provides a method to control the norm of the potential by its geodesic distance in the critical submanifold.

Recently, Donaldson \cite{MR2507220}\cite{MR2433928}\cite{MR2154300}\cite{MR1988506} proved the following theorem.
\begin{thm}(Donaldson \cite{MR2507220})
If a polarized complex toric surface with zero Futaki invariant is
$K$-stable, then it admits a cscK metric.
\end{thm}
Following the same line of Donaldson's theorem, one may ask
\begin{ques}(Chen \cite{Chencom})\label{cong csck}
Let $M$ be a polarized complex toric surface with zero Futaki invariant. If $M$ is $K$-stable, does the pseudo-Calabi flow converge to a cscK metric up to holomorphic diffeomorphisms?
\end{ques}
Actually, according to \thmref{intro: stability} the convergence up to holomorphic diffeomorphisms in Question \ref{cong csck} implies the exponential convergence.
\bibliography{bib}
\bibliographystyle{plain}
\end{document}